\documentclass[12pt, reqno]{amsart}
\usepackage{amsfonts,amssymb,latexsym,amsmath, amsxtra}
\usepackage[all]{xy}
\usepackage[dvips]{graphics}

\usepackage[OT2,T1]{fontenc}
\DeclareSymbolFont{cyrletters}{OT2}{wncyr}{m}{n}
\DeclareMathSymbol{\Sha}{\mathalpha}{cyrletters}{"58}

\theoremstyle{plain}
\newtheorem{theorem}{Theorem}[section]
\newtheorem{corollary}[theorem]{Corollary}

\newtheorem{lemma}[theorem]{Lemma}
\newtheorem{proposition}[theorem]{Proposition}
\theoremstyle{definition}
\newtheorem{definition}[theorem]{Definition}
\theoremstyle{remark}
\newtheorem*{remark}{Remark}

\numberwithin{equation}{section}

\newcommand{\ZZ}{{\mathbb Z}}
\newcommand{\RR}{{\mathbb R}}

\newcommand{\HH}{{\mathbb H}}

\newcommand{\CC}{{\mathbb C}}

\newcommand{\sF}{{\mathcal F}}

\newcommand{\R}{{\bf R}}

\newcommand{\SL} {{\rm SL}}

\usepackage{times}
\usepackage[T1]{fontenc}
\usepackage{mathrsfs}
\usepackage{color}

\newcommand{\Z}{\mathbb Z}
\newcommand{\C}{\mathbb C}

\def\matr#1#2#3#4{\left(\begin{array}{cc}#1&#2\\#3&#4\end{array}\right)}

\def\cF{\mathcal F}

\def\H{\mathbb H}

\def\cF{\mathcal F}

\def\({\left(}
\def\){\right)}
\def\<{\left<}
\def\>{\right>}

\newcommand{\ol}[1]{\overline{{#1}}}
\newcommand{\wt}[1]{\widetilde{#1}}

\newcommand{\im}[1]{\text{Im}\(#1\)}

\newcommand{\abs}[1]{\left|#1\right|}

\newcommand{\htE}{{\widehat{E}}}

\newcommand{\dhtE}{{\widehat{\widehat{E}}}}
\newcommand{\htz}{{\widehat{\zeta}}}
\newcommand{\tM}{{\tilde{M}}}

\def\PSL{\text{PSL}}

\newcommand{\smatr}[4]{\(\begin{smallmatrix} #1 & #2 \\ #3 & #4\end{smallmatrix}\)}

\newcommand{\sgn}[1]{{\rm sgn}(#1)}

\begin{document}

\title[Polyharmonic Maass Forms]{Polyharmonic Maass Forms for $\PSL(2,\ZZ)$}

\author{Jeffrey C. Lagarias}
\address{Dept. of Mathematics, The University of Michigan, Ann Arbor, MI 48109-1043, USA}
\email{lagarias@umich.edu}

\author{Robert C. Rhoades}
\address{Center for Communications Research, Princeton, NJ 08540}
\email{rob.rhoades@gmail.com}

\thanks{Research of the first author was partially supported by
NSF grants DMS-1101373 and DMS-1401224.
Research of the second author was partially supported by an
NSF Mathematical Sciences Postdoctoral Fellowship.}

\date{July 30, 2015}
\thispagestyle{empty} \vspace{.5cm}


\begin{abstract}
We discuss  polyharmonic Maass forms of even integer weight on $\PSL(2,\Z)\backslash \H$,
which are a generalization of classical Maass forms. 
We explain the role of the real-analytic Eisenstein series $E_k(z,s)$ and the differential operator 
$\frac{\partial}{\partial s}$ in this theory. 
\end{abstract}

\maketitle

\begin{center}
\emph{  
Dedicated to the memory of Marvin Knopp}
\end{center}

\section{Introduction}\label{sec:Intro}
Classical Maass forms of even integer weight $k\in 2\Z$  for $\PSL(2, \Z)$ are smooth functions on $\H:= \{ z=x+iy \in \C : \Im(z) >0\}$ such that 
\begin{enumerate}
\item[(1)] 
(modular invariance condition)
$f(z) = f \big |_{k} \gamma(z)$ for each $z \in \H$ and $\gamma = \smatr{a}{b}{c}{d} \in \PSL_2(\Z)$ where 
the \emph{slash operator} of weight $k\in \Z$ is defined by 
$$ g\big |_k  \gamma(z) = (cz+d)^{-k} g\( \frac{az+b}{cz+d}\).$$
That is,
$f(z)$ satisfies
$$f( \frac{az+b}{cz+d}) = (cz+d)^k f(z),
$$ for $\gamma \in \PSL_2(\Z).$
\item[(2)]
(Laplacian eigenfunction condition)
$f$ satisfies $(\Delta_k-\lambda) f = 0$ for some $\lambda \in \C$, where 
$$\Delta_k := y^2 \( \frac{\partial^2}{\partial x^2} + \frac{\partial^2}{\partial y^2}\) 
- i ky\(\frac{\partial}{\partial x} + i \frac{\partial}{\partial y}\)$$
is the weight $k$ hyperbolic Laplacian.
(We follow the convention of Maass \cite{Maa83}
for the sign of the Laplacian. Many authors call $-\Delta_k$
the hyperbolic Laplacian, for example \cite{DIT13}.)
\item[(3)]
(moderate growth condition)
There exists an $\alpha \in \R$ such that $f(x+iy) = O(y^\alpha)$ as $y\to \infty$,
uniformly in $x \in \RR$. 
\end{enumerate}
Such forms have played an important role in number theory and  automorphic forms,
see for instance, Section 1.9 of \cite{Bu97}. More generally, we may 
consider forms of integer weight $k \in \Z$,
but nothing is gained since all odd integer weight forms on $PSL(2, \ZZ)$ must vanish identically, due to 
the weight $k$ modular invariance condition  applied to $\gamma = \pm
\smatr{0}{-1}{1}{0}$.

In this paper we study the situation where the eigenvalue
$\lambda=0$, which is the case corresponding to holomorphic modular forms,
where  we relax the Laplacian eigenfunction condition to require only that  the  forms satisfy 
\begin{equation}\label{poly-h-eqn}
\(\Delta_k\)^m f (z)= 0
\end{equation}
 for some non-negative integer $m$. 
 We call such  $f(z)$   {\em polyharmonic Maass forms}, in parallel with the literature
on polyharmonic functions for the Euclidean Laplacian,
on which there is an extensive literature,
starting around 1900 (see Almansi \cite{Al1899}, Aronszajn, Crease and Lipkin \cite{ACL83}, Render \cite{Ren08}). 
The integer parameter $m$  in \eqref{poly-h-eqn} might  be termed the {\em  (harmonic) order} 
in parallel with the literature on polyharmonic functions.   
We will instead use the term {\em harmonic depth} 
because  the term  ``order" is 
used  in   conflicting ways in the literature\footnote{The partial differential equations  literature
assigns order $2m$ to    $\(\Delta_k \)^m$  
 when treated in terms of the differential
operators $\frac{\partial}{\partial x}$ and $\frac{\partial}{\partial y}$. A second conflict is that the 
term $p$-harmonic Maass form is  by Bruggeman \cite{Brug14}  with a different meaning,
 to refer to a function annihilated by the  {\em $p$-Laplacian}, in which $p$ corresponds to the weight parameter $k$ in our notation.}.
 Moreover we  allow  the harmonic depth 
  to take half-integer
 values, as follows.  We assign  to any nonzero holomorphic modular form  $f(z)$
 the harmonic depth $\frac{1}{2}$, 
 because  it is annihilated  by  the
d-bar operator $\frac{\partial}{\partial \bar{z}}$ which is ``half" of the harmonic Laplacian 
$\Delta_k =\( y^2\frac{\partial}{\partial z} + 2iky \)\frac{\partial}{\partial \bar{z}}$.
In addition, to any weight $k$ Maass form $f(z)$ such that $\(\Delta_k\)^{m} f(z)= g(z)$
with $g(z)$ a nonzero holomorphic modular form 
we assign
 harmonic depth $m + \frac{1}{2}$. 

 In a companion paper \cite{LR15S} we will treat functions satisfying the more general equation
\begin{equation}\label{poly-h-eqn2}
( \Delta_k- \lambda)^m f (z)= 0
\end{equation}
 for some non-negative integer $m$, where $\lambda \in \CC$ with $\lambda \ne 0$.
 We call functions satisfying \eqref{poly-h-eqn2} {\em shifted polyharmonic functions}
  with {\em eigenvalue shift $\lambda$}. If they transform as modular forms of weight $k$
  then  we call them  {\em shifted polyharmonic depth  Maass forms}. We let $V_k^m(\lambda)$ 
  denote the vector space of all such weight $k$ forms with eigenvalue $\lambda$ 
  on $PSL(2, \ZZ)$ that have moderate growth at the cusp. 
 For $\lambda \ne 0$ we refer to the minimal $m$ annihilating such a function as its {\em (shifted) harmonic depth};
 it is always an integer since holomorphic forms do not occur.
 From this more general perspective the  case $\lambda=0$ is exceptional
  in allowing  forms having half-integer depth. 

In this paper we  determine the spaces $V_k^m(0)$ 
 for all even integer weights $k$ for the full modular group $PSL(2, \ZZ)$,
 showing that it is finite-dimensional  and explicitly exhibiting a full set
 of basis elements.  The finite dimensionality of these spaces has entirely to do with the moderate growth condition; if this is relaxed, then 
the resulting space of solutions can be infinite dimensional. Our main observation is that the 
new members of the vector spaces in the harmonic depth $m$ case aside from  those  the harmonic depth $1$ case
involve  derivatives in the $s$-variable of nonholomorphic Eisenstein series 
$E_k(z, s)$, evaluated at $s=0$.

Note that modular forms with poles at cusps are
never of moderate growth; thus weakly holomorphic modular forms are not included in the classes of
forms we study.   Bruggeman \cite{Brug14} has recently considered 
larger classes of  weight $k$ Maass forms for 
the full modular group $SL(2, \ZZ)$, allowing 
linear exponential growth at the cusps, for the $\Delta_k$-operator.

  By the term classical Maass form, we mean that our definition includes classical holomorphic modular forms
  for all weights $k \in 2\ZZ$.
The original treatment of  Maass \cite{Maa53}, \cite[Chap. IV]{Maa83} and  
other later treatments  (Bump \cite[Sect. 2.1]{Bu97}, Duke et al \cite[Sect. 4]{DFI02}) use a different (and generally inequivalent)
definition. They  define {\em (original) weight $k$ Maass forms}, denoted  here $f^{[M]}(z)$, to require instead of  (1), (2), (3)  the conditions
(1'), (2') , (3) with 
\begin{enumerate}
\item[(1')] 
(modified modular invariance condition)
For each $z \in \H$ and $\gamma = \smatr{a}{b}{c}{d} \in \PSL_2(\Z)$ 
there holds
$$ f^{[M]}(z)  = j_{\gamma}(z)^{-k} f^{[M]}\( \frac{az+b}{cz+d}\),$$
with multiplier system
$$
j_\gamma(z) := \frac{ cz+d}{|cz+d|} =e^{i \arg(cz+d)}.
$$
\item[(2')]
(modified Laplacian eigenfunction condition)
$f^{[M]}$ satisfies $(\Delta_k^{[M]}-\lambda) f^{[M]} = 0$ for some $\lambda \in \C$, where 
$$
\Delta_k^{[M]} := y^2 \( \frac{\partial^2}{\partial x^2} + \frac{\partial^2}{\partial y^2}\) 
- i ky\frac{\partial}{\partial x}
$$
is the  modified  weight $k$ hyperbolic Laplacian, $\Delta_k^{[M]} = \Delta_k - ky \frac{\partial}{\partial y}.$
\end{enumerate}
The two definitions coincide for weight $k=0$ but differ otherwise. The multiplier $j_{\gamma}(z)$ is
not holomorphic, so for $k \ne 0$ original weight $k$ Maass forms  are never holomorphic functions of $z$.
Given a classical Maass form $f(z)$ of  weight $k$, the 
associated function 
$f^{[M]}(z) := y^{\frac{k}{2}} f(z)$ is an original weight $k$ Maass form
and vice versa, 
so one can  easily transfer results between the two definitions.
We use  classical weight $k$ Maass forms  
because holomorphic modular forms 
of all weights $k \in 2\ZZ$ play a special role in our results.  

\subsection{Polyharmonic Maass forms}\label{sec11}

 Forms satisfying  conditions (1), (3) of our Maass form definition 
  but with the Laplacian eigenfunction condition (2) relaxed to $\Delta_k^m f  = 0$ 
 are called {\em $m$-harmonic Maass forms} of weight $k$ for $\PSL(2, \Z)$.   
Let $V_{k}^m:= V_k^m(0)$ denote the vector space of such functions, 
where $m \in \frac{1}{2} \ZZ$ with $m \ge 1/2$.
 It is known that this space is finite dimensional, without explicitly determining the dimension,
 see  \cite[Theorem 8.5]{Bo97}.
The object of this paper is to give an explicit construction of a basis of $V_k^m(0)$, and to  explain the special role of 
the non-holomorphic Eisenstein series in this construction.

The non-holomorphic Eisenstein series $E_0(z, s)$ of weight $0$  for  $\PSL(2, \ZZ)$, is given by 
\begin{equation}~\label{100}
E_0(z, s) := 
\frac{1}{2}\left( \sum_{(m,n) \in \ZZ^{2} \backslash (0,0)} \frac{y^s}{|mz+ n|^{2s}}\right) =
\frac{1}{2} \zeta(2s) \left(\sum_{{(c,d) \in \Z^2} \atop{ (c,d)=1}} \frac{y^s}{|cz+ d|^{2s}}\right).
\end{equation}  
with  $z= x+iy \in \HH$ and $\Re(s) > 1$.
The completed weight $0$ Eisenstein series is 
$$
\htE_0(z,s) := \pi^{-s} \Gamma(s) E_0(z,s).
$$
For each $z \in \HH$ the completed series   is known to analytically continue to a meromorphic function of $s$ 
 that satisfies the functional equation
$$
\htE_0(z, s)= \htE_0(z, 1-s).
$$
The singularities of $\htE_0(z, s)$ are simple poles at $s=0$ and $s=1$. 
For our purposes it is better to work with the  {\em doubly-completed non-holomorphic Eisenstein series} 
\begin{equation}~\label{eqn:doublycompletedWt0}
\dhtE_0(z,s) = s(s-1)\pi^{-s} \Gamma(s)E_0(z,s)
\end{equation}
which removes the two poles  and maintains the symmetry between $s$ and $1-s$ .
See  Theorem \ref{th38} for details. 
We define the Taylor coefficients with respect to $s$ by $F_n(z)$, namely
\begin{equation}
\dhtE_0(z,s) = \sum_{n=0}^\infty F_n(z) s^n.
\end{equation}

%
%

\begin{theorem}\label{thm:wt0VectorSpace}
For integer $m \ge 1$ the  vector space $V_0^m(0)$ of weight 0 $m$-harmonic Maass forms 
has $V_0^m (0)= V_0^{m-1/2}(0)$ and is $m$-dimensional. 
 Moreover, the set $\{ F_0(z), \cdots, F_{m-1}(z) \}$ is a basis for $V_0^m(0)$. 
\end{theorem}

We call this set the {\em Taylor basis}. For example, the space $V_0^{1/2}(0)$ is one-dimensional
and is spanned by
the constant functions.

Automorphic forms, including Maass forms, are related to forms of different weights via differential operators called 
the (Maass) weight  raising and lowering operators. See, for instance, Section 2.1 of \cite{Bu97}.   
In the theory of harmonic Maass forms there is a closely related operator which moves harmonic forms of weight 
$k$ to forms of weight $2-k$.  This operator was introduced by Bruinier and Funke \cite{BF04} and is given by 
\begin{equation}\label{eqn:xi}
\xi_k := 2i y^k \ol{\frac{\partial}{\partial \ol{z}}}.
\end{equation}
In other words, if $f: \H \to \C$ is a smooth function then $\xi_k(f) = 2i y^k \ol{\frac{\partial}{\partial \ol{z}} f},$
using the Wirtinger derivative 
 $\frac{\partial}{\partial \ol{z}}= \frac{1}{2} (\frac{\partial}{\partial x} + i \frac{\partial}{\partial y}).$

Hence, in the context of studying weight $0$ harmonic forms it is natural to also introduce and study weight $2$ forms.  
We define the completed real-analytic weight $2$ Eisenstein series for $Re(s)>1$ by 
\begin{equation}\label{301}
\htE_2(z,s) = \pi^{-(s+1)} \Gamma(s+1) E_2(z,s)=\pi^{-(s+1)} \Gamma(s+1) 
\left(\frac{1}{2}\sum_{{(m,n) \in {\ZZ}^2 \backslash (0,0)} } \frac{y^s}{(mz+n)^2 |mz+ n|^{2s}}\right),
\end{equation}
and  the {\em doubly-completed real-analytic weight $2$ Eisenstein series}   by
\begin{equation}\label{eqn:doublycompletedWt2}
\dhtE_2(z,s):= (s+1)s \pi^{-(s+1)} \Gamma(s+1) \zeta(2s+2) \left(\frac{1}{2}\sum_{{(c,d)} \atop{ (c,d)=1}} \frac{y^s}{(cz+d)^2 |cz+ d|^{2s}}\right).
\end{equation}
Parallel to  the case of $\dhtE_0(z,s)$ this series has an analytic continuation to $s\in \C$ (see 
Theorem \ref{th38}).
Define its Taylor coefficients by 
\begin{equation}
\dhtE_2(z,s) = : \sum_{n=0}^\infty G_n(z) s^n.
\end{equation}
In analogy with Theorem \ref{thm:wt0VectorSpace} we have the following.
%
%
%
\begin{theorem}\label{thm:wt2VectorSpace}
For integer $m \ge 0$ the  vector space $V_2^m(0)$ of weight 2, $m$-harmonic Maass forms
has $V_2^{m}(0) = V_2^{m+ 1/2}(0)$ and  is $m$-dimensional.  
Moreover, $G_0(z) \equiv 0$ and the set $\{ G_1(z), \cdots, G_{m}(z)\}$ is a basis for $V_2^m(0)$. 
\end{theorem}

Again we call this a {\em Taylor basis} of $V_2^m(0)$. For example, the space $V_2^{1/2}(0) = \{ 0\}$.

To explore the relationship between these weight  0 and weight 2 forms it is convenient to introduce symmetrized versions of the 
Taylor coefficients $\{F_j\}$ and $\{G_j\}$. 
Define the {\em symmetrized Taylor basis} functions  

\begin{align}
\label{eqn:symmetric_F} \widetilde{F}_n(z) :=& (-1)^n \Big(F_n(z) +
 \sum_{\ell=1}^n { n+\ell \choose n} F_{n-\ell}(z)\Big)  \\
 \label{eqn:symmetric_G} 
\widetilde{G}_n(z):=&  \,G_n(z) + \sum_{\ell=1}^n (-1)^{\ell} { n+\ell \choose n} G_{n-\ell}(z).
\end{align}
These functions are obtained by  a triangular  change of basis from the Taylor bases
given above.

%
%

\begin{theorem}\label{thm:ladder}
For the symmetrical basis functions \eqref{eqn:symmetric_F} and \eqref{eqn:symmetric_G}, we have 
\begin{enumerate}
\item 
$\Delta_0 \wt{F}_n(z) = \wt{F}_{n-1}(z)$ 
and $\xi_{0} \wt{F}_{n}(z)  = \wt{G}_{n}(z)$
\item 
$\Delta_2 \wt{G}_n(z) = \wt{G}_{n-1}(z)$
   and $\xi_{2} \wt{G}_{n}(z) = \wt{F}_{n-1}(z)$.
\end{enumerate}
\end{theorem}

This theorem can be summarized by the  picture in Figure 1.

\begin{figure}[h]
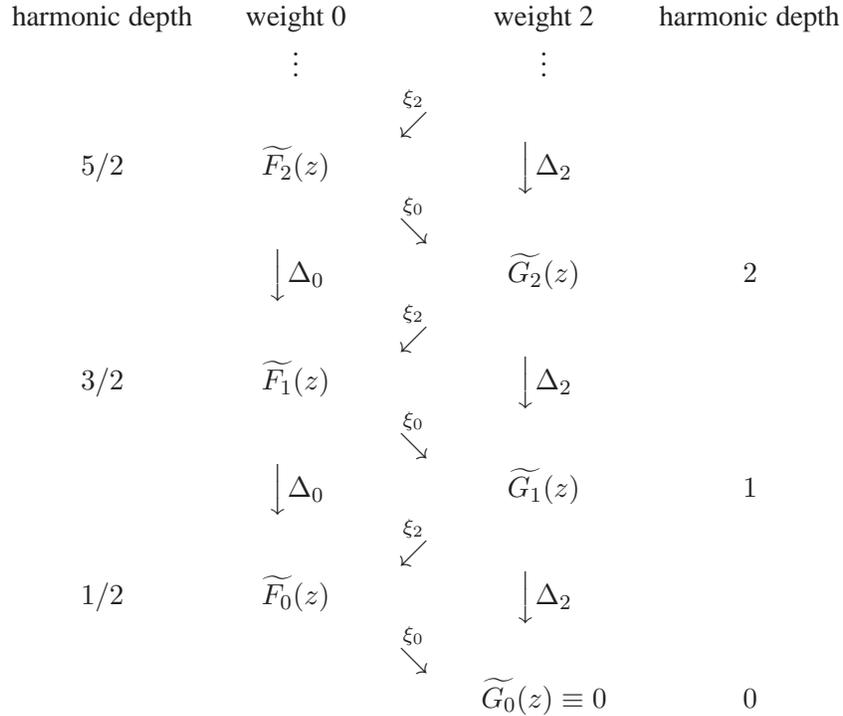

{\small 
\begin{equation*}\label{eqn:ladder}
\begin{array}{ccccccccc}
\mbox{harmonic depth}&&\mbox{weight 0} && &&\mbox{weight 2} && \mbox{harmonic depth}\\
&&\vdots &&  && \vdots &&\\
&&&& \stackrel{\xi_2}{  \swarrow} &&&&\\
5/2 &&\widetilde{F_2}(z) &&  &&  \Big \downarrow \Delta_2&&\\
&& && \stackrel{\xi_0}{  \searrow}  &&&&\\ 
&&\Big \downarrow \Delta_0 & & & & \wt{G_2}(z) && 2\\
&&&& \stackrel{\xi_2}{  \swarrow}  &&&&\\
3/2 &&\widetilde{F_1}(z) &&  &&  \Big \downarrow \Delta_2 &&\\
&& && \stackrel{\xi_0}{  \searrow}  &&&&\\ 
&&\Big \downarrow \Delta_0 & & & & \wt{G_1}(z) && 1\\
&&&& \stackrel{\xi_2}{  \swarrow}  &&&&\\ 
1/2 && \wt{F_0}(z)  &&  && \Big \downarrow \Delta_2&& \\
 &&&& \stackrel{\xi_0}{  \searrow}  &&&&\\ 
&&  & & & & \wt{G_0}(z) \equiv 0  && 0 \\
\end{array}
\end{equation*}
}
\caption{Tower and ramp structure for weights $0$ and $2$}
\end{figure} \label{fig1}

One may view this picture as a set of two ``towers"
interconnected by a set of ``ramps" which are given by
the operators  $\xi_k$ and $\xi_{2-k}$.
The  ``ramp"  structure incorporates  well-known factorizations of the unshifted Laplacians
\begin{equation}\label{eqn:factor}
\Delta_k = \xi_{2-k}  \xi_k,
\end{equation}
which themselves move down the towers.
The bottom elements of the towers of symmetrized Taylor bases are well-known. We have
\begin{align}
\wt{F}_0(z) =& \frac{1}{2} \label{eqn:F0} \\
\wt{F}_1(z) =& -\frac{1}{2} \gamma +\log(4 \pi) + \log (\sqrt{y}|\Delta(z)|^{\frac{1}{12}}),
\end{align}
where $\gamma$ is Euler's constant. The second of these formulas is essentially equivalent to 
the ``$s=0$"
version of Kronecker's first limit formula for the Eisenstein series $E_0(z,s)$.
In addition we have 
\begin{align}
\wt{G}_0(z) =& 0 \label{eqn:G0} \\
\wt{G}_1(z) =&-\frac{\pi}{6} 
-\frac{1}{2y}  +4 \pi \big( \sum_{n=1}^{\infty}\sigma_1(n)e^{2 \pi i n z}\big)
\end{align}
which is related to the  Fourier series expansion\footnote{The calculation here currently differs
in the sign of $\frac{1}{2y}$ from \cite[(3.7) ]{DIT13}.}
 for $E_2(z, s)$.   
In the above formulas
  $\Delta(z)$ is the weight $12$ cusp form for $\SL(2, \Z)$ given by 
$\Delta(z) = e(z ) \prod_{n=1}^\infty \( 1- e(nz)\)^{24}$, where $e(x) := e^{2\pi i x}$,
and $\sigma_1(n) = \sum_{d\mid n} d$.
See Lemmas  \ref{le27} and \ref{le37}   for details. 


Complementing this ``tower" structure  given by the Laplacians $\Delta_k$ is the operator
$\frac{\partial}{\partial s} $, which  allows movement up the ``tower''.  
More precisely,   the Taylor coefficients are given by
 $F_n(z) = \frac{1}{n!}\, \frac{\partial^n}{\partial s^n} \dhtE(z,s) \mid_{s=0}.$
To move up the ``tower'' one level at a time, set 
$$
 {\widehat{\widehat{E_0^{(n)}}}}(z,s) := \frac{1}{n!} \(\frac{\partial}{\partial s}\)^n \dhtE_0(z,s).
$$
Then we have $F_n(z) = {\widehat{\widehat{E_{0}^{(n)}}}}(z,0)$ and
 $$F_{n+1}(z) = \frac{1}{(n+1)} \frac{\partial}{\partial s} \widehat{\widehat{E_0^{(n)}}}(z,s) \big |_{s=0}.$$
Hence, there  a sort of ``switching'' between the differential operator $\frac{d}{ds}$
and the Laplacian operator $\Delta_{k}$ moving down the tower.

There are a number of references in the literature where the differentiation with respect to $s$ has appeared in 
the context of Eisenstein series and automorphic forms. For instance, the works of Kudla-Rapoport-Yang \cite{Ku04,KRY04,KY10, yang04, yangSecond},
  Duke-Imamoglu-Toth \cite{DIT1, DIT2} 
 and Duke and Li \cite[page 2]{DL15} contain such results.  

While harmonic Maass forms have played a significant role in the emerging theory of
Ramanujan's mock theta functions (see \cite{onoCD}, for example), there are 
few examples of $m$-harmonic
Maass forms for $m>1$.  One exception is the appearance of 
a 2-harmonic form in the work of Bringmann-Diamantis-Raum \cite{BDR}. Their work  
is related to non-critical  modular $L$-values.

\subsection{Shifted  and arbitrary integer weight polyharmonic Maass forms}\label{sec12}
In Section \ref{sec:Weight} we state
corresponding results 
for  polyharmonic Maass forms (i.e. $\lambda=0$) for all even integer 
weights $k$ and all harmonic depths $m \ge \frac{1}{2}$.
This  case is  complicated by the presence of holomorphic cusp forms, which appear at
harmonic depth  $1/2$. 
We determine the dimension for each $(k,m)$.  An important feature  is 
Proposition \ref{prop:liftOfCuspForms}, which implies that any nonzero holomorphic
cusp form is not the  image under the operator $\xi_{2-k}$ of any  harmonic depth 
$m =3/2$ polyharmonic Maass form, hence not the image 
under  $\Delta_k$ of a harmonic depth $2$
Maass form. 
This result  shows that  polyharmonic Maass forms 
are of  value in  understanding 
 classical modular forms. Namely, for $k \ge 4$ it 
 provides a   characterization of  holomorphic
Eisenstein series  viewed inside the vector space $M_k$ of holomorphic modular forms,
as follows.

\begin{proposition}\label{prop:14}
 For even weight $k \ge 4$ the one-dimensional space of holomorphic Eisenstein series $E_k^1(0) \subset V_k^1(0)$ is  the range 
 of the Laplacian $\Delta_k$ acting on the space $V_k^2(0)$  of 
depth $2$ polyharmonic Maass forms having  moderate growth at the cusp. 
\end{proposition}

 As already mentioned, a companion paper  \cite{LR15S} treats  shifted polyharmonic Maass
 forms satisfying $(\Delta_k -\lambda)^m f(z) =0$ for a fixed $\lambda \ne 0$.
  The case of   general $\lambda$ includes  Maass cusp forms,
  and via nonholomorphic
  Eisenstein series also has a connection with the Riemann zeta zeros,
  as observed in 1981 by Zagier \cite{Za81a}.
  The harmonic depth
 of such $f$ is always an integer $m$. 
The analogues of Figures 1 and 2 of this paper
 for $\lambda \ne 0$  have  ``towers"  with actions of  Laplacians
 $\Delta_k$ and $\Delta_{2-k}$,
however the ``ramps" are replaced by  rungs of a ``ladder"
at each fixed depth $m$, with the depth preserved by the
 action of  the operators $\xi_k$ and $\xi_{2-k}$.

\subsection{Roadmap}\label{sec13}
Section \ref{sec:Weight} states further main results 
 for polyharmonic Maass forms of arbitrary even integer weight
other than $0$ or $2$. 
Section \ref{sec:Background} 
contains known results on classical holomorphic modular forms 
(Section \ref{sec:Holomorphic}),
and basic facts on non-holomorphic Eisenstein series,
including their Fourier expansions and functional equations (Section \ref{sec:NHE}).
Section \ref{sec:newsec4} formulates results on  the Fourier expansions of polyharmonic
Maass forms, some based on \cite{LR15S}.
Section \ref{sec:newsec5} presents results about the 
Bruinier-Funke 
nonholomorphic differential operator $\xi_k$, showing
that it preserves spaces of functions of moderate growth at the cusp.
Section \ref{sec:newsec6} establishes for $\lambda =0$ that holomorphic cusp forms 
are not  the image under $\Delta_k$ of any bi-harmonic Maass form.
Section \ref{sec:7} computes the action of $\xi_k$ on non-holomorphic
Eisenstein series. 
 Section \ref{sec:TS} gives recursions for the Taylor series coefficients
 of $\dhtE_k(z; s_0)$ in the $s$-variable, which are functions of $z$.
 It shows these coefficients are polyharmonic Maass forms, 
 and determine  recursion relations they satisfy under the action of $\Delta_k$
 and $\xi_k$. Section \ref{sec:new9}  constructs modified bases of $V_k^m(0)$ 
 which satisfies simpler recursion relations with respect to $\Delta_k$ and $\xi_k$.
 Section \ref{sec:Proofs} presents proofs  of the main theorems of Section \ref{sec:Intro}, namely 
 Theorems \ref{thm:wt0VectorSpace}, \ref{thm:wt2VectorSpace}, and \ref{thm:ladder}. 
 Section \ref{sec:Proofs2}  completes the proofs of the main results of Section \ref{sec:Weight}.

\subsection{Notation}

We define the {\em completed Riemann zeta function} by
$$
\htz(s) := \pi^{-s/2} \Gamma(s/2) \zeta(s).
$$
For a fixed $\gamma = \smatr{a}{b}{c}{d}$ we write $\gamma \cdot z = \frac{az+b}{cz+d}$. 

\section{General even integer weight polyharmonic Maass forms}\label{sec:Weight}
Section \ref{sec:Intro} presented results dealing with weight $0$ and weight $2$
 polyharmonic Maass forms. 
This section describes the analogous picture for arbitrary even integer weights. The scenario for arbitrary weight 
is complicated by  the presence of cuspidal holomorphic modular forms. 

\begin{definition}
A holomorphic modular form for $\SL_2(\Z)$ of weight $k \in \Z$ is a holomorphic function 
$f: \H \to \C$ satisfying 
$f \mid_k \gamma(z) = f(z)$, i.e. 
$$(cz+d)^{-k} f\( \frac{az+b}{cz+d}\) = f(z) \ \ \text{ for all } \ \ \smatr{a}{b}{c}{d} \in \SL_2(\Z)$$
and 
$f(z) = O(y^A) \ \ \text{ as } \ \ y \to \infty \ \ \text{ for some } A.$
Moreover, if $f(z) = O(y^{-k/2})$ as $y\to \infty$ $f$ is called a cusp form.  
\end{definition}
The set of all holomorphic modular forms is a vector space, denoted $M_k$. 
Moreover, let $S_k$ denote the space of all such cusp forms. With our convention
on half-integer harmonic depth we have
$M_k =V_k^{\frac{1}{2}}$.

For $z\in \H$, an even integer $k$,  and $Re(s) >1$ define for $Re(s) > 1-k$ 
the (non-holomorphic) Eisenstein series
\begin{equation}\label{eqn:E_k_def_I}
E_k(z, s) := \frac{1}{2} \left( \sum_{\gamma = \smatr{a}{b}{c}{d}  \in \Gamma_\infty \backslash \SL_2(\Z)} \psi_s \mid_k \gamma(z)  \right)=
\frac{1}{2} \zeta(2s) \left(\sum_{{(c,d) \in \Z^2} \atop{ (c,d)=1}} \frac{y^s}{(cz+d)^{k}|cz+ d|^{2s}}\right),
\end{equation}
where  $\psi_s(z) = \im{z}^s$.  For $k \ge 4$ an even integer the value $s=0$ 
gives the unnormalized
holomorphic Eisenstein series
\begin{equation}\label{holo-ES}
E_k(z, 0) :=\frac{1}{2}  \sum_{(m,n) \in \ZZ^2 \backslash (0,0)} \frac{1}{(mz+n)^k}.
\end{equation}
The  Eisenstein series $E_k(z, s)$ meromorphically continues in the $s$-variable with a functional equation given
in terms of the  {\em completed non-holomorphic Eisenstein series}
\begin{equation}
\htE_k(z, s) := \pi^{-(s + \frac{k}{2})} \Gamma(s +\frac{k}{2} + \frac{|k|}{2}) E_k(z, s).
\end{equation}
The functional equation is
\begin{equation*}
\htE_k(z, s) = \htE_k(z, 1-k-s),
\end{equation*}
which has critical line $Re(s) = \frac{1-k}{2}.$ The function $\htE_k(z, s)$
is an entire function of $s$ for $k \ne 0$ and for $k=0$ it has simple poles at $s=0, 1$,
see Theorem \ref{th38}.

Define the {\em doubly completed (non-holomorphic) Eisenstein series}
\begin{equation}\label{eqn:E_k_def}
\dhtE_k(z,s) := 
(s+\frac{k}{2}) (s +\frac{k}{2}-1) \htE_k(z, s)
\end{equation}
which is an entire function of $s$.
By convention we denote the Taylor coefficients at $s=0$ for $k \in 2\ZZ$ by 
\begin{equation}\label{eqn:E_k_taylor_coefficients}
\dhtE_k(z,s) = \begin{cases} \sum_{n=0}^\infty F_{n, k} (z) s^n  & \mbox{for weights} \,\,  k \le 0, \\ 
 \sum_{n=0}^\infty G_{n,k}(z) s^n & \mbox{for weights} \,\,  k \ge 2.\end{cases}
 \end{equation}
 For  weights $0$ and $2$ we have
  $F_{n,0}(z) = F_n(z)$ and $G_{n, 2}(z) = G_n(z)$ as given in
 Section \ref{sec:Intro}.
The cases $k=0, 2$ are distinguished by the property  that  
 the factor $(s+\frac{k}{2})(s+ \frac{k}{2}-1)$ in $\dhtE_k(z, s)$ has a zero at $s=0$,
leading to different behavior than the general case; this is one
reason we have treated them separately.  The  ``ramp" and  ``tower" structure is slightly
altered in the general case, as pictured below.

The following two  theorems contains results for general even integer weights $k \ne 0, 2$ which parallel Theorems \ref{thm:wt0VectorSpace}, 
\ref{thm:wt2VectorSpace}, and \ref{thm:ladder}. 

\begin{theorem}\label{thm:arbitraryweight}
Let $m \ge 1$ be an integer, and $k \in 2\ZZ$ an even integer, with $k \ne 0$ or $2$.
\begin{enumerate}
\item[(1)]
 For an even integer $k\le -2$, $V_{k}^m(0)$ is $m$-dimensional. Moreover, 
$\{ F_{0, k}(z), \cdots, F_{m-1, k}(z)\}$ is a basis for $V_{k}^m(0)$.
In this case $V_k^{1/2}(0) = \{0\}$
and
$\dim(V_k^m(0)) = \dim (V_k^{m+ \frac{1}{2}}(0))$
for all $m \ge 0$.
\item[(2)]
For an even integer weight $k \ge 4$, 
$$V_k^m(0) = E_k^m(0)  + S_k $$
where $E_k^m(0)$ is an $m$-dimensional subspace spanned by 
$\{ G_{0, k}(z), \cdots, G_{m-1, k}(z)\}$. 
Moreover, $S_k$ consists of cusp forms 
and has dimension $\max([k/12]-1,0)$ 
if $k\equiv 2\pmod{12}$ and $[k/12] $ if $k\not\equiv 2 \pmod{12}$ and $[x]$ 
is the largest integer less than or equal to $x$. 
In this case
$V_k^{1/2}(0) = M_k$ has positive dimension, 
 containing the holomorphic Eisenstein series $G_{0,k}(z)$, and
 $\dim(V_k^m(0)) = \dim (V_k^{m- \frac{1}{2}}(0))$
 for all $m \ge 1$. 
\end{enumerate}
\end{theorem}

The next result describes the ``ramp" and ``tower" structure for weights $k \ge 4$
paired with dual weights $2-k \le -2$ using modified basis functions.
\begin{theorem}\label{thm:arbitrary-ramp}
Let $m \ge 1$ be an integer, and $k \in 2\ZZ$ an even integer, with $k \ge 4$.
Then
\begin{enumerate}
\item There are modified basis functions  $\wt{G}_{n,k}(z)$ and $\wt{F}_{n,2-k}(z)$ such that the following
are true for each even integer $k\ge 2$.
\begin{enumerate}
\item {\rm (``Ramp" relations)} 
$$\xi_{k} \wt{G}_{n,k}(z) =(k-1) \wt{F}_{n-1,k}(z) \quad \mbox{and} \quad \xi_{2-k} \wt{F}_{n, 2-k}(z)= \wt{G}_{n, 2-k}(z).$$ 
\item {\rm  (``Tower" relations)} 
$$
\Delta_k \wt{G}_{n, k} (z)= (k-1)\wt{G}_{n-1,k}(z) \quad \mbox{and} \quad
 \Delta_{2-k} \wt{F}_{n, 2-k}(z) =  (k-1)\wt{F}_{n-1, 2-k}(z).
 $$
\end{enumerate}
\item One choice of modified functions takes the form
\begin{eqnarray*}
\wt{G}_{n,k}(z) &= & G_{n, k}(z) +\sum_{\ell=1}^n \frac{1}{(k-1)^{\ell}}{{n+\ell}\choose{n}}
G_{n-\ell,k}(z)\\
\wt{F}_{n,2-k}(z) & = & (-1)^n \Big( F_{n, k}(z) + 
(-1)^n \sum_{\ell=1}^n (-1)^{\ell}
\frac{1}{(k-1)^{\ell}}{{n+\ell}\choose{n}}F_{n-\ell,k}(z)\Big).
\end{eqnarray*}
\end{enumerate}
\end{theorem}

The modified basis functions are not unique; their form is classified in Theorem \ref{thm:76}.  Figure  
\ref{fig2}
pictures schematically the ``tower" and ``ramp" structure for 
 weights $k \ge 4$ paired with  dual weights $2-k \le -2$.
 
\begin{figure}
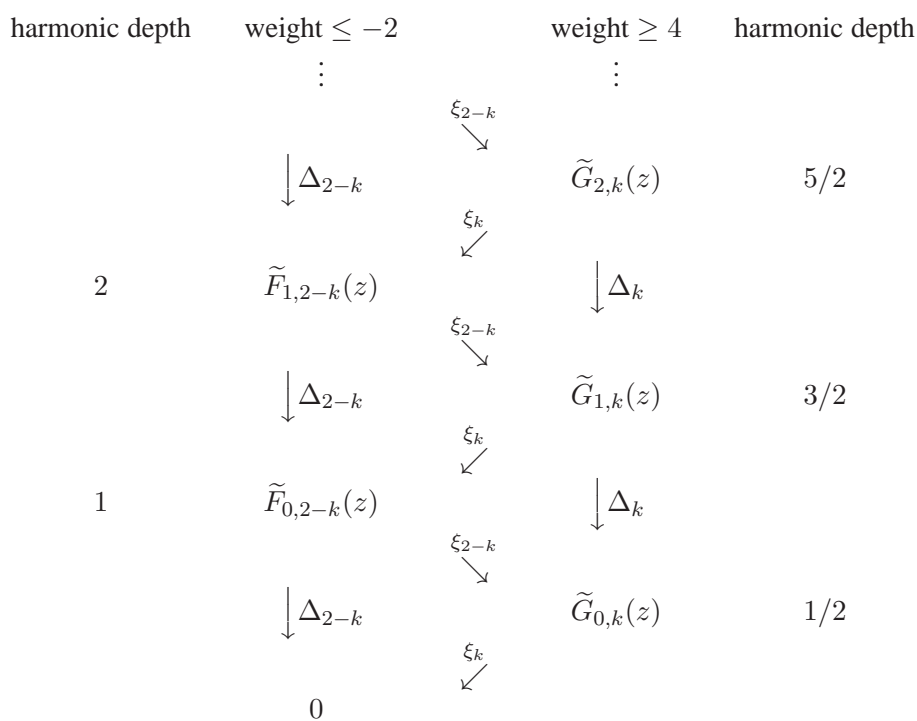

{\small 
\begin{equation*}\label{eqn:ladder2}
\begin{array}{ccccccccc}
\mbox{harmonic depth}&&\mbox{weight  $ \le -2$} && &&\mbox{weight $\ge 4$} && \mbox{harmonic depth}\\
&&\vdots &&  && \vdots &&\\
&& && \stackrel{\xi_{2-k}}{  \searrow}  &&&&\\ 
&&\Big \downarrow \Delta_{2-k} & & & & \wt{G}_{2,k}(z) && 5/2\\
&&&& \stackrel{\xi_k}{  \swarrow} &&&&\\
2 &&\widetilde{F}_{1, 2-k}(z) &&  &&  \Big \downarrow \Delta_k&&\\
&& && \stackrel{\xi_{2-k}}{  \searrow}  &&&&\\ 
&&\Big \downarrow \Delta_{2-k} & & & & \wt{G}_{1,k}(z) && 3/2\\
&&&& \stackrel{\xi_k}{  \swarrow}  &&&&\\
1 &&\widetilde{F}_{0, 2-k}(z) &&  &&  \Big \downarrow \Delta_k &&\\
&& && \stackrel{\xi_{2-k}}{  \searrow}  &&&&\\ 
&&\Big \downarrow \Delta_{2-k} & & & & \wt{G}_{0,k}(z) && 1/2\\
&&&& \stackrel{\xi_k}{  \swarrow}  &&&&\\ 
 && 0\,\, &&  &&&& \\
&&  & & & &  && \\
\end{array}
\end{equation*}
}
\caption{Tower and ramp structure for weights $k$ and $2-k$ for $ k \ge 4$}\label{fig2}
\end{figure}

For even weights $k \ge 4$  the  holomorphic Eisenstein series 
$\wt{G}_{0,k}(z) = c_k E_k(z)$
occurs at harmonic depth $1/2$.
This  picture is  reversed from the case of weight $k=2$, paired with dual weight $0$.
These results are  proved in Section \ref{sec:Proofs2}.

\section{Background} \label{sec:Background}

This section recalls 
well known results on holomorphic modular forms for $\SL_2(\Z)$,
 and presents results  on Maass's Eisenstein series as discussed in 
Maass \cite{Maa83}, and  parallel properties for the
non-holomorphic Eisenstein series $E_k(z, s)$ considered here.


\subsection{Classical Holomorphic Modular Forms for $\SL_2(\Z)$}\label{sec:Holomorphic}
This section collects some well known results concerned with 
holomorphic modular forms for $\SL_2(\Z)$. 
We reference Zagier's article 
\cite{Za08}, however there are many good references for these results. 

\begin{theorem}[Corollary p.15 of \cite{Za08}]\label{thm:classicalDimension}
For $k < 0$ and for $k$ odd, $\dim(M_k) = 0$. 
For $k\ge 0$ even 
$$\dim(M_k) = \begin{cases} \left[ \frac{k}{12} \right] + 1 & k \not\equiv 2 \pmod{12} \\ 
     \left[ \frac{k}{12} \right]  & k \equiv 2 \pmod{12} \end{cases}$$
where $[x]$ is the largest integer less than or equal to $x$. 
\end{theorem}

A source of holomorphic modular forms are the Eisenstein series. For $k \in 2 \ZZ$ with $k \ge 4$
the (normalized) holomorphic Eisenstein series is
\begin{equation}\label{eqn:holomorphic_eisenstein}
E_{k}(z) 
= \frac{1}{2} \sum_{\begin{subarray}{c} c, d \in \Z \\ (c,d) = 1 \end{subarray}}  \frac{1}{(cz+d)^{k}}.
\end{equation}
Clearly $E_{k}(z) = 1 + o(1) \ne 0$ as $y\to \infty$. 

\begin{proposition}[Proposition 4, p.16 of \cite{Za08}]\label{prop:classicalFreelyGenerated}
The ring of holomorphic modular forms $M_\star  = \bigoplus_{k \in 2\ZZ} M_k$
for $SL(2, \ZZ)$  is freely  generated 
under multiplication by $E_4(z)$ and $E_6(z)$. 
Moreover, $M_k = E_k + S_k$ with $S_k$ being the space of cusp forms and  $\dim(E_k) = 1$
for $k=0$ and $k \ge 4$ and is zero otherwise. 
\end{proposition}
To parallel computations in later sections we record the Fourier series of $E_k(z)$. 

\begin{proposition}[Proposition 5, p.16 of \cite{Za08}]
The Fourier expansion of the Eisenstein series $E_k(z)$ for an even integer $k \ge 4$  is given by
$$E_k(z) = \frac{(2\pi i)^k}{\zeta(k) (k-1)!} \( - \frac{B_k}{2k} + \sum_{n=1}^\infty \sigma_{k-1}(n) q^n\),$$
where $\sigma_\ell(n) = \sum_{d\mid n, d>0} d^\ell$, $\zeta$ is the Riemann zeta function, 
and $B_k$ is the $k$th Bernoulli number defined by the generating function $\sum_{k=0}^\infty B_k \frac{x^k}{k!} = \frac{x}{e^x -1}$. 
\end{proposition}

Finally we recall the notion of a  \emph{weakly holomorphic modular form}.
\begin{definition}\label{de34}
A {\em weakly holomorphic modular form} $f:\H \to \C$ on $\SL_2(\Z)$  of weight $k$ 
is a holomorphic  function  satisfying 
$f\mid_k \gamma(z) = f(z)$ for all $\gamma \in \SL_2(\Z)$ such that $f$ has at most linear exponential growth 
as $y\to \infty$ (i.e. $f(iy) = O(e^{A y})$ for some $A$ as $y\to\infty$). 
We let  $M_k^{!}$ denote the vector space
of weakly holomophic modular forms of weight $k \in 2\ZZ$ for $SL(2, \ZZ)$.
\end{definition}

Such functions may have poles at the cusp.
A basic example of a weakly holomorphic modular form of weight $0$ 
having such a pole is the modular invariant 
$j(z) =  1728\frac{E_{4}(z)^3}{E_4(z)^3 - E_6(z)^2}$. 
The vector space  $M_k^{!}$ is infinite-dimensional for all $k \ge 0$.


\subsection{Properties of  Non-holomorphic Eisenstein series}\label{sec:NHE}

 Recall from \eqref{eqn:E_k_def_I} that
for $2k \in 2\ZZ$  the  nonholomorphic Eistenstein series $E_k(z, s)$ is
\begin{equation}\label{Eis}
E_{k}(z, s) :=\frac{1}{2} \sum_{(m,n) \in \ZZ^2 \backslash  (0,0)} \frac{y^s}{|mz+n|^{2s} (mz+n)^{k}}.
\end{equation}
This series converges absolutely for  $Re(s)  > 1-k$. The resulting function 
transforms under  elements of $SL(2, \ZZ)$ as 
$$
E_{k}\left(\frac{az+b}{cz+d}, s\right) =  (cz+d)^{2k} E_k(z, s) \text{ when } \matr{a}{b}{c}{d} \in \SL(2, \Z)
$$
For  even weights $k \ge 4$ at $s=0$  
this function specializes to  an
unnormalized version\footnote{With our scaling $E_{k}(z, 0) = \frac{1}{2}G_k(z)$,
where $G_k(z)$ is the usual unnormalized holomorphic Eisenstein series in 
Serre \cite{Se73}. }
of the  holomorphic Eisenstein series $E_k(z)$, 
given in Section \ref{sec:Holomorphic}, with
$$
E_{k}(z, 0) = \frac{1}{2} \sum_{(m,n) \in \ZZ^2 \backslash 0} \frac{1}{(mz+n)^2}
= \frac{1}{2} \zeta(k) E_k(z).
$$

Using results on Maass's nonholomorphic Eisenstein series $G(a, \ol{z}; \alpha, \beta)$, as
treated in \cite{LR15S}, we obtain for 
the completed non-holomorphic Eisenstein series
 $\htE_k(z, s)= \pi^{- \frac{s+k}{2}}\Gamma(s+ \frac{k}{2} + |\frac{k}{2}|) E_k(z, s)$
the following result.
%
%

\begin{proposition}\label{prop:FourierExpansionArbitrary}
{\rm (Fourier expansion of $\htE(z,s)$)}
For $k\in 2\ZZ$,
the completed nonholomorphic Eisenstein series 
$\htE(z, s)$
has the Fourier expansion 
\begin{align*}
\htE_k(z,s) = C_{0}(y,s) 
& + (-1)^{\frac{k}{2}} (\sqrt{2} \pi)^{2s+k} \pi^{-s-\frac{k}{2}} \Gamma\( s 
+\frac{k}{2}
+ \frac{\abs{k}}{2}\)   \\
&\ \ \  \times \sum_{n\in \Z \backslash \{0\}} 
\frac{\sigma_{2s+k-1}(n)}{\Gamma\( s+\frac{k}{2}(1 + \sgn{n}) \) } (2\pi \abs{n}y)^{-\frac{k}{2}} 
 W_{\frac{1}{2} \sgn{n} k , s+\frac{k-1}{2} } (4\pi \abs{n} y)e^{2\pi i n x},
 \end{align*}
in which $W_{\kappa, \mu}(z)$ denotes the Whittaker $W$-function and 
the constant term is
\begin{align*}
C_0(y,s) =& \left(   \frac{\Gamma\(s+\frac{k}{2} + \frac{\abs{k}}{2}\)}{\Gamma\(s+\frac{k}{2}\)}  \htz(2s+k)\, y^s\right. \\
& \ \ \ \left.+  (-1)^{\frac{k}{2}} \frac{\Gamma\( s+ \frac{k}{2}\) \Gamma \( s+ \frac{k}{2} + \frac{\abs{k}}{2}\)}{\Gamma( s+ k) \Gamma(s) } \htz(2-2s-k)\, y^{1-s-k}\right).
    \end{align*}
 and $\sigma_s(n) = \sum_{d |n, d>0} d^{s}.$
\end{proposition}

\begin{proof}
This result is established 
 in \cite{LR15S}. 
The Whittaker $W$-function is described in Whittaker and Watson \cite[Sect. 16.12]{WW27},
and   \cite{LR15S},
as well as the handbook \cite{AS64} and its sequel \cite{OLBC10}.
\end{proof}

We now collect various analytic properties of the Eisenstein series.
%
%

\begin{theorem}\label{th38}
{\rm (Properties of $E_k(z,s)$)}
Let $k \in 2\ZZ$.

(1) {\rm (Analytic Continuation)} For fixed $z \in \HH$, the  completed weight $k$ Eisenstein series  
$$
\htE_k(z, s) := \pi^{-(s+\frac{k}{2})}\Gamma(s+\frac{k}{2}+ \frac{|k|}{2}) E_k(z,s)
$$
analytically continues to the $s$-plane as
 a meromorphic function. For $k=0$ 
its has two singularities, which are 
simple poles at $s=0$ and $s=1$
with residues $-\frac{1}{2}$ and $\frac{1}{2}$, respectively.
For $ k \ne 0$ it is an entire function.

(2) {\rm (Functional Equation)} For fixed $z \in \HH$, the completed weight
$k$ Eisenstein series
satisfies the functional equation
\begin{equation}~\label{213aa}
\htE_k(z, s)= \htE_k(z, 1-k-s).
\end{equation}
The doubly-completed series
$$\dhtE_k(z,s) :=(s+ \frac{k}{2})(s +\frac{k}{2}-1) \htE_k(z,s)$$
 is an entire function of $s$ for all $k \in 2\ZZ$
and  satisfies the same functional equation
\begin{equation}~\label{213b}
\dhtE_k(z, s)= \dhtE_k(z, 1-k-s).
\end{equation}
The center line of these functional equations is $Re(s) = \frac{1-k}{2}.$

(3) {\rm ( $\Delta_k$-Eigenfunction )} $E_k(z, s)$ is a (generalized) eigenfunction
of the non-Euclidean Laplacian operator $\Delta_k$ with eigenvalue $\lambda= s(s+k-1).$ 
That is, for all $s \in \CC$,
\begin{equation}\label{209a}
\Delta_k E_{k}(z, s) = s(s+k-1) E_{k}(z, s).
\end{equation}
This eigenfunction property holds for 
the completed functions $\htE_k(z, s)$ and $\dhtE_k(z, s).$
\end{theorem}

\begin{proof} 
This result is given 
in \cite{LR15S}.
\end{proof}

We note an interesting consequence:   we may obtain the doubly-completed
Eisenstein series from the singly-completed one by applying a
differential operator, a  shifted Laplacian. This supplies an analytic construction
of the ``double completion". 
%
%

\begin{corollary}\label{cor39}
For $k \in 2\ZZ$
the doubly-completed non-holomorphic Eisenstein series $\htE(z, s)$
is obtainable from the completed Eisenstein series $\htE(z, s)$ by
\begin{equation}\label{209cc}
(\Delta_k +\frac{k^2}{4}) \htE_{k}(z, s) = \dhtE_k(z, s).
\end{equation}
\end{corollary}

\begin{proof}
By Theorem \ref{th38} (3) we  have
$$
(\Delta_k +\frac{k^2}{4}) \htE_{k}(z, s)= 
\big(s(s+k-1) + \frac{k^2}{4}\big) \htE(z, s) =
 (s+ \frac{k}{2}) (s+ \frac{k}{2}-1)\htE(z, s) =  \dhtE(z, s).
$$
\end{proof}

\section{Polyharmonic Fourier Series } \label{sec:newsec4}

We formulate results on the form of Fourier series of (shifted) polyharmonic functions, some taken
from \cite{LR15S}. We  give results  Fourier series for general $\lambda=s(s+k-1)$,
 although this paper considers only  $\lambda=0$,  because our results require taking derivatives
 with respect to $s$.

\subsection{Harmonic Fourier coefficients}\label{sec:40Harm}

The following result gives
the allowable functional  form of the Fourier coefficients for periodic functions in the hyperbolic plane
that satisfy $(\Delta_k- \lambda) h_n(z) e^{2\pi i n x}  = 0$
and have moderate growth at the cusp, meaning  $O(y^c)$ for some finite $c$ as $y \to \infty$.  This result is due to Maass \cite[Hilfssatz 6]{Maa53}
and involves Whittaker $W$-functions, as treated in \cite[Sect. 16.12]{WW27}.

\begin{theorem}\label{th40}
{\rm (Harmonic Fourier Coefficients of Moderate Growth)}
Let  $k \in 2 \ZZ$ and suppose  that $f_n(z) = h_n(y) e^{2 \pi i nx}$  is a
shifted-harmonic  function  for $\Delta_k$ on $\HH$ with eigenvalue $\lambda \in \CC$, i.e.
it satisfies
$$
(\Delta_k- \lambda) f_n(z) =0 \quad \mbox{for all} \quad z=x+iy \in \HH.
$$
Write $\lambda = s_0(s_0+k-1)$ for some $s_0 \in \CC$ (there are generally two choices for $s_0$.)
Suppose also that $f_n(z)$ has at most polynomial growth in $y$ at the cusp.
Then the  complete set of such functions $h_n(y)$, are given in the following list.
\begin{enumerate}
\item[(1)]
Suppose  $n \ne 0$, and let  $\epsilon = \sgn{n} \in \{\pm 1\}$. Then
$$
h_n(y) = y^{-\frac{k}{2}}  
 \Big(a_0 (W_{\frac{1}{2}\epsilon k, s_0+ \frac{k-1}{2} } )(4 \pi |n| y)\Big)
$$
for some constant $a_0 \in \CC$, with $W_{\kappa, \mu} (z)$ 
denoting a  $W$-Whittaker function. 
\item[(2a)]
Suppose $n =0$ with $s_0 \ne \frac{1-k}{2}$ (equivalently, with  $\lambda \ne -(\frac{1-k}{2})^2$). Then
$$ 
h_n(y) =  a_0^{+}  y^{s_0}+ a_0^{-} y^{1-k-s_0}
$$
for some constants $a_0^{+}, a_0^{-} \in \CC$.
\item[(2b)]
Suppose $n =0$ with $s_0 = \frac{1-k}{2}$ (equivalently,   with $\lambda = -(\frac{1-k}{2})^2$).
Then
$$
h_n(y) =  \sum_{j=0}^1 \frac{\partial^j}{\partial s^j} (y^{s})|_{s=s_0}
= \sum_{j=0}^{1} a_j (\log y)^j y^{\frac{k-1}{2}}.
$$
for some constants $a_0, a_1\in \CC$.
\end{enumerate}
\end{theorem}

\begin{proof}
For $n \ne 0$ this follows from  Maass \cite[Lemma 6, Chap. 4, p. 181]{Maa83}.
For $n =0$ it is a simple calculation.
\end{proof}

\subsection{Polyharmonic Fourier coefficients}\label{sec:FourierExp}

We state  a result proved in  \cite{LR15S} which specifies the allowable form of
 individual Fourier coefficients $h_n(z)$ which satisfy $(\Delta_k- \lambda)^m h_n(z)e^{2\pi i n x} = 0$
 and  have moderate growth at the cusp. 
 This result is  based on the fact  that polyharmonic functions in $z$ 
in the Fourier coefficients 
are obtainable by repeated partial derivatives
 $\frac{\partial}{\partial s}$ of the eigenfunctions with respect to the eigenvalue parameter $s$. 
 This fact holds because
 although the operator $\frac{\partial}{\partial s}$ commutes
 with the Laplacian $\Delta_k$, it
 does not commute with $\Delta - s(s+k-1)I$ but instead satisfies the commutator
 identity
 $$
 [\Delta - s(s+k-1)I,\frac{\partial}{\partial s}] = (1-k-2s)I.
 $$
 of Heisenberg type. 

\begin{theorem}\label{th41}
{\rm (Polyharmonic Fourier Coefficients)}
Let  $k \in 2 \ZZ$ and suppose  that $f_n(z) = h_n(y) e^{2 \pi i nx}$  is a
shifted-polyharmonic  function  for $\Delta_k$ on $\HH$ with eigenvalue $\lambda \in \CC$, i.e.
it satisfies
$$
(\Delta_k- \lambda)^m f_n(z) =0 \quad \mbox{for all} \quad z=x+iy \in \HH.
$$
Write $\lambda = s_0(s_0+k-1)$ for some $s_0 \in \CC$ (there are generally two choices for $s_0$.)
Suppose also that $f_n(z)$ has at most polynomial growth in $y$ at the cusp.
Then the  complete set of such functions $h_n(y)$, are given in the following list.
\begin{enumerate}
\item[(1)]
Suppose  $n \ne 0$, and let  $\epsilon = \sgn{n} \in \{\pm 1\}$. Then
$$
h_n(y) = y^{-\frac{k}{2}} \Big( 
\sum_{j=0}^{m-1} a_j  \frac{\partial^j}{\partial s^j} (W_{\frac{1}{2}\epsilon k, s+ \frac{k-1}{2} } ) \big|_{s=s_0}(4 \pi |n| y)
\Big)
$$
for some constants $a_j \in \CC$, with $W_{\kappa, \mu} (z)$ 
denoting the $W$-Whittaker function. 
\item[(2a)]
Suppose $n =0$ with $s_0 \ne \frac{1-k}{2}$ (equivalently, with  $\lambda \ne -(\frac{1-k}{2})^2$). Then
$$ 
h_n(y) = \sum_{j=0}^{m-1} a_j^{+} \frac{\partial^j}{\partial s^j}( y^{s}) \big|_{s=s_0}+
 \sum_{j=0}^{m-1} a_j^{-} \frac{\partial^j}{\partial s^j} (y^{1-k-s}) \big|_{s=s_0}
 = \sum_{j=0}^{m-1} a_j^{+}  (\log y)^j y^{s_0} + \sum_{j=0}^{m-1}  a_j^{-}  (-\log y)^j y^{1-k-s_0}
$$
for some constants $a_j^{+}, a_j^{-} \in \CC$.
\item[(2b)]
Suppose $n =0$ with $s_0 = \frac{1-k}{2}$ (equivalently,   with $\lambda = -(\frac{1-k}{2})^2$).
Then
$$
h_n(y) = \sum_{j=0}^{2m-1} a_j \frac{\partial^j}{\partial s^j} (y^{s}) \big|_{s=s_0}
= \sum_{j=0}^{2m-1} a_j (\log y)^j y^{\frac{k-1}{2}}.
$$
for some constants $a_j \in \CC$.
\end{enumerate}
\end{theorem}

\begin{remark}\label{remark:43}
\begin{enumerate}
\item[(1)]  For each  $\lambda \in \CC$ 
there are two choices of $s$ except $\lambda = -(\frac{1-k}{2})^2$,
where there is a unique choice. These  choices
of $s$  correspond to the variable switch $(\kappa, \mu)$ to $(\kappa, - \mu)$
which corresponds to $s \mapsto 1-(k+s)$. Either one of the choices  leads
to a basis  of  the same vector space of functions. 

\item[(2)] If no growth conditions are imposed on $f_n(z)$ at the cusp then, when $n \ne 0$, 
 additional terms are allowed in the Fourier expansion, which involve
   a suitable linearly independent Whittaker function,  call it
 $\tM_{\frac{\epsilon k}{2}, s + \frac{k-1}{2}}(y)$, and its derivatives with respect to $s$, 
    see \cite{LR15S} for further discussion.
  The Whittaker $W$-functions $W_{\kappa, \mu}(y)$ have modulus going to $0$ exponentially fast 
  as $y \to \infty$, while the indpendent functions $\tM_{\kappa, \mu}(y)$ have modulus growing
  exponentially fast in $y$ as the real variable $y \to \infty$.
 \end{enumerate}
\end{remark}

\subsection{Polyharmonic Fourier series expansions}\label{sec42}


Theorem \ref{th41} 
 implies a  Fourier expansion formula valid for  all $m$-harmonic Maass forms
 with shifted eigenvalue $\lambda$.  We 
 introduce  a new notation for these functions. For $n \ne 0$ set $\epsilon = \frac{n}{|n|} \in  \{ \pm 1\}$,
 and for each $m \ge 0$  set
 \begin{equation}\label{u-emm2}
u_{\epsilon, k,|n|}^{[m], -}(y; s_0) :=
y^{-\frac{k}{2}} \frac{\partial^{m}}{\partial s^{m}}(W_{\kappa, s+ \frac{k-1}{2}})(4 \pi |n|y)|_{s=s_0}.
\end{equation}
(Here the $-$ superscript refers to this function having rapid decay at the cusp;  there is
an independent  solution $u_{\epsilon, k,|n|}^{[m], +}(y; s_0)$ having rapid growth at the cusp, cf. \cite{LR15S}.)
For $n=0$ with $s_0 \ne \frac{1-k}{2}$,
for each $m \ge 0 $ set
\begin{eqnarray*}
u_{ k,0}^{[m], +}(y; s_0) &:=& \frac{\partial^{m}}{\partial s^{m}}y^s \big|_{s=s_0}= \quad (\log y)^m\, y^{s_0}\\
u_{ k, 0}^{[m],-}(y; s_0) &:= &\frac{\partial^{m}}{\partial s^{m}}y^{1-k-s} \big|_{s=s_0} 
= (-1)^m (\log y)^m y^{1-k-s_0}.
\end{eqnarray*}
For $n=0$ with $s_0 = \frac{1-k}{2}$  set
 \begin{eqnarray*}
u_{k,0}^{[m], +} (z; \frac{1-k}{2}) &=&(\log y)^{2m-2} y^{\frac{1-k}{2}},\quad\mbox{and} \quad
u_{k,0}^{[m], -} (z; \frac{1-k}{2}) = (\log y)^{2m-1} y^{\frac{1-k}{2}}.
\end{eqnarray*}
Then we have the following result.

%
%

\begin{theorem}\label{lem:abstractFourier}
{\rm (Fourier expansion in $V_k^m(\lambda)$)}
Let $f(z) \in V_k^m(\lambda) $ for some $k \in 2\ZZ$.
Let  $m \ge 1$,
and fix an $s \in \CC$ with $\lambda= s(s+k-1)$. 
Then the Fourier expansion of $f(z)$ exists and  has the form
$$
f(z) =  \sum_{j=0}^{m-1} \big(c_{0, j}^{+} u_{k, 0}^{[j],+}(y)  + c_{0,j}^{-} u_{k, 0}^{[j],-}(y)  \big)
  + \sum_{\epsilon \in\{ \pm 1\}}\Big( \sum_{n=1}^{\infty} \sum_{j=0}^{m-1}  c_{\epsilon, j, n}^{-} u_{\epsilon, k, n}^{[j], -}(y) e^{ 2\pi i (\epsilon n)x}\Big),
  $$
 in which  $c_{\epsilon, j}^{\pm}$ and $c_{\epsilon, j, n}^{-}$ are constants. 
 This Fourier expansion converges absolutely and uniformly to $f(z)$ on compact subsets of $\HH$.
\end{theorem}
\begin{proof}
A weight $k$ modular form on $PSL(2, \ZZ)$ 
has $f(z) = f(z+1)$, using $\gamma =\smatr{1}{1}{0}{1}$,
whence it has a Fourier expansion of the form 
$f(z) \sim \sum_{n\in \Z} h_n(y) e^{2\pi i n x}$
with coefficient functions
$$h_n(y) = \int_{0}^1 f(z) e^{-2\pi i n x} dx.$$ 
Because  $f(z)$ has moderate growth of order $O(y^K)$ as $y\to \infty$,
for some fixed finite $K$, 
we conclude that each $h_n(y)$ separately has moderate growth of the same order. Additionally,  
the shifted polyharmonic condition $(\Delta_k -\lambda)^m f(z) = 0$  implies that 
 each of its Fourier coefficients separately satisfy
$$
(\Delta_k- \lambda)^m \( h_n(y) e^{2\pi i n x} \) =0.
$$
This fact holds by separation of variables, since
the application of   $(\Delta_k- \lambda)$ to any function  $f_1(y) e^{2\pi i n x}$ yields
another function $f_2(y) e^{2 \pi i n x}$, whence 
$(\Delta_k- \lambda)^m\(h_n(y) e^{2\pi i n x} \)= \tilde{h}_n(y) e^{2 \pi i nx}$, whence
$$
0= (\Delta_k- \lambda)^m (f(x+iy)) = \sum_{n \in \ZZ} (\Delta_k- \lambda)^m\(h_n(y) e^{2\pi i n x} \) 
= \sum_{n \in \ZZ} \tilde{h}_n(y) e^{2 \pi i nx}.
$$
The uniqueness of Fourier series expansions of a real-analytic function then implies that all $\tilde{h}_n(y) =0$.
Finally 
Theorem \ref{th41} applies to each $f_n(z)$ separately to give an expansion of the given form.

Finally the absolute and uniform convergence on compact subsets follows from the 
known real-analyticity
of such $f(z)$.
\end{proof}

%
\subsection{$1$-harmonic  Fourier expansions}\label{sec43}
We   give a  second version of the Fourier expansion for the 
special case $m=1$ and eigenvalue $\lambda=0$,
involving incomplete Gamma functions.  It  is a special case of a Fourier
expansion for  $1$-harmonic Maass forms that appears in the literature.
We relate it to the  Fourier expansion version given above 
in Theorem \ref{lem:abstractFourier} in terms of Whittaker functions.

%
%

\begin{lemma}\label{lem:FourierExp1Harmonic} 
For any $k \in 2 \ZZ$, if  $f(z) \in V_k^1(0)$, then it has a Fourier expansion of form
\begin{equation}\label{gamma-fourier}
f(z) = \sum_{n=1}^{\infty}  b_{-n} \Gamma(1-k, 4 \pi |n|y) e^{-2\pi i nz}
+ \big( b_0 y^{1-k} + a_0\big) + \sum_{n =1}^{\infty} a_n e^{2 \pi i n z}
\end{equation}
in which $\Gamma(\kappa, y) = \int_{y}^\infty t^{\kappa -1} e^{-t} dt$ denotes the  incomplete Gamma function.  
\end{lemma}

\begin{proof}
Since $f(z)$ has moderate growth
the Fourier expansion of  Theorem \ref{lem:abstractFourier} applies, taking $m=1$.
Thus we have
$$
f(z) = y^{-\frac{k}{2}} \Big( \sum_{n=1}^{\infty} \tilde{b}_{-n} W_{-\frac{k}{2}, \frac{1-k}{2}}(4 \pi |n|y) e^{-2 \pi i n x}
+ \big( \tilde{b}_0 y^{1-k} + \tilde{a}_0\big) +
\sum_{n=1}^{\infty} \tilde{a}_n W_{\frac{k}{2}, \frac{1-k}{2}}( 4 \pi |n| y) e^{2 \pi i n x} \Big).
$$
for certain complex constants $\tilde{a}_n, \tilde{b}_{-n}$ for all $n \ge 0$.

We assert that for $n \le -1$, 
\begin{equation}\label{claim1}
(4 \pi |n|y)^{-\frac{k}{2}}W_{-\frac{k}{2}, \frac{1-k}{2}}(4 \pi |n|y)= 
\Gamma(1-k, 4 \pi |n|y) e^{- 2 \pi n y}=
\Gamma(1-k, 4 \pi |n|y) e^{2 \pi |n| y}.
\end{equation}
To prove the assertion we use the identity \cite[(13.18.5)]{NIST}
$$
z^{\frac{1}{2} - \mu} e^{\frac{z}{2}}\Gamma(2 \mu, z)= W_{\mu - \frac{1}{2}, \mu}(z),
$$
in which we take $\mu= \frac{1-k}{2}$ and $z= 4 \pi |n|y$ to obtain
$$
W_{-\frac{k}{2},\frac{1-k}{2}}(4 \pi |n|y)= (4 \pi |n|y)^{\frac{k}{2}} e^{2 \pi |n| y} \Gamma (1- k, 4 \pi |n|y),
$$
from which \eqref{claim1} follows.

We next assert that  for $n \ge 1$ we have
\begin{equation}\label{claim2}
(4 \pi ny)^{-\frac{k}{2}}W_{\frac{k}{2}, \frac{1-k}{2}}(4 \pi ny)= e^{- 2 \pi n y}.
\end{equation}
To show this  we use the 
identity \cite[(13.4.31)]{NIST}
$$
W_{\kappa, \mu} (z) = W_{\kappa, -\mu}(z).
$$
and the identity \cite[(13.18.17)]{NIST} valid for integer $n \ge 0$ and integer $\alpha$ that
$$
W_{\frac{\alpha+1}{2} +n, \frac{\alpha}{2}}(z) = (-1)^n n! e^{-\frac{z}{2}} z^{\frac{\alpha +1}{2}}L_n^{(\alpha)}(z),
$$
with $L_n^{(\alpha)}(z)$ being a (modified) Laguerre polynomial of degree $n$. We choose $n=0$ to obtain
$$
W_{\frac{\alpha+1}{2}, -\frac{\alpha}{2}}(z)=W_{\frac{\alpha+1}{2}, \frac{\alpha}{2}}(z)=  e^{-\frac{z}{2}} z^{\frac{\alpha +1}{2}},
$$
since $L_0^{(\alpha)}(z) =1.$ 
We take $\alpha=k-1$ and $z= 4 \pi n y$ to obtain \eqref{claim2}.

Combining the two assertions, the 
coefficients of the new Fourier expansion \eqref{gamma-fourier} are
related to the old by $a_0 = \tilde{a}_0$,
$b_0 = \tilde{b}_0$, and for all $n \ge 1$,  
\begin{equation}\label{fourier-convert}
a_n = (4 \pi n)^{\frac{k}{2}}\, \tilde{a}_n \quad \mbox {and} \quad
 b_{-n} = (4 \pi |n|)^{\frac{k}{2}} \,\tilde{b}_{-n}. 
\end{equation}
\end{proof}
%
\begin{remark}
A more general form of this formula for weak Maass forms appears in 
  Bruinier and Funke \cite{BF04}, see Section \ref{sec61}.
(Compare also  the discussion in
  the introduction of \cite{DIT13}.)
  Bruinier and Funke 
note in their equations (3.2a) and (3.2b)  and the sentence before those equations 
 that   any $f(z) \in M_k^{!}$
 has a Fourier expansion of the form
\begin{equation}\label{eq413}
f(z) = b_0y^{1-k} + 
\sum_{n\gg 0} a_n e^{2\pi i n z} + \sum_{n \ll 0} b_n \Gamma(1-k, 4\pi \abs{n} y) e^{2\pi i n z}.
\end{equation}
The condition of at most polynomial growth made above
imposes the strengthened restrictions on the ranges of summation in the two parts.
\end{remark}

\section{The $\xi$-Operator} \label{sec:newsec5}

\subsection{ Properties of the differential operators $\xi_k$}\label{sec:DifferentialOperators}

Bruinier and Funke \cite[Proposition 3.2]{BF04} introduce the operator
$$
\xi_k(f)(z) = 2i y^k \overline{\frac{\partial}{\partial \bar{z}}f(z)},
$$
which is related to the Maass raising and lowering operators.

We state  basic lemmas concerning  the differential operators $\xi_k$,
mentioned in the introduction to \cite{DIT13}.

%
%

\begin{lemma}\label{lem:xi_commute}
Let $f : \H \to \CC$ be a $C^{\infty}$-function of two
real variables $(x, y)$, and set $z=x+iy$. Then  for any $\gamma = \smatr{a}{b}{c}{d} \in \SL_2(\R)$ 
and for any integer $k$, 
$$
\xi_k \( (cz+d)^{-k} f( \gamma \cdot z)\) = (cz+d)^{k-2} (\xi_k f)(\gamma \cdot z).
$$
\end{lemma}

\begin{proof}
This is a  calculation, see for example \cite{LR15S}. 
\end{proof}

Lemma \ref{lem:xi_commute} implies that if $f(z)$ is a weight $k$ 
(holomorphic or non-holomorphic) modular form for a discrete subgroup $\Gamma$
of $SL(2, \ZZ)$, with no growth conditions imposed on any cusp, then $\xi_k f$ is a
weight $2-k$  modular form for $\Gamma$, again imposing no growth conditions at any cusp.

%
%
\begin{lemma}\label{lem:DeltaFactor}
The operator $\Delta_k=y^2 \( \frac{\partial^2}{\partial x^2} + \frac{\partial^2}{\partial y^2}\) - i ky\(\frac{\partial}{\partial x} + i \frac{\partial}{\partial y}\)$
 factorizes as
$$
\Delta_k = \xi_{2-k} \xi_k. 
$$
Thus $\Delta_{2-k} = \xi_{k} \xi_{2-k}$.
\end{lemma}

\begin{proof}
This is an easy calculation. 
\end{proof}

\subsection{Action of $\xi_k$ on Fourier coefficient functions}\label{sec52}

One can directly compute $\xi_k$-operator action on the Fourier 
coefficient functions  in Theorem \ref{th40}.
The  operator $\xi_k$ maps weight $k$ forms  to weight $2-k$ forms, takes  Fourier coefficient $n$ to Fourier coefficient $-n$,
and converts the holomorphic parameter  $s$ to the anti-holomorphic parameter $-\bar{s}$.


%
%

\begin{lemma}\label{le54}

(1) Let $n \le -1$. Then
$$
\xi_k\big( y^{-\frac{k}{2}} W_{-\frac{k}{2}, s + \frac{k-1}{2}}(4 \pi |n|y) e^{2 \pi i n x}\big) =
- y^{-(\frac{2-k}{2})} W_{\frac{2-k}{2}, -\bar{s} + \frac{1-k}{2}}(4 \pi |n|y) e^{-2\pi i nx}.
$$

(2) Let $n \ge 1$. Then
$$
\xi_k\big( y^{-\frac{k}{2}} W_{\frac{k}{2}, s + \frac{k-1}{2}}(4 \pi ny) e^{2 \pi i n x}\big) =
\overline{s}(1-k - \overline{s}) y^{-(\frac{2-k}{2})} W_{-(\frac{2-k}{2}), -\bar{s} + \frac{1-k}{2}}(4 \pi |n|y) e^{-2\pi i n x}.
$$

(3) There holds
$$ 
\xi_k( y^s) = \overline{s} y^{-1+k +\overline{s}}
\quad
\mbox{and} \quad
\xi_k( y^{1-(s+k)}) = (1 -k- \overline{s}) y^{-\overline{s}}.
$$
\end{lemma}

\begin{proof}
These results follow by a calculation, given in detail in \cite{LR15S}.
\end{proof}


\subsection{$\xi_k$ preserves moderate growth}\label{sec54}

%
%
\begin{lemma}\label{lem42}\label{lem:subspaceXi}
The action of the $\xi$ operator on shifted polyharmonic vector spaces preserves
the moderate growth property for all $\lambda$. That is, 
\[
\xi_k( V_k^m(0)) \subseteq  V_{2-k}^m(0) .
\]
\end{lemma}

\begin{proof}
Now $f(z) \in V_k^m(0)$ has $(\Delta_k)^m f(z) =0.$
Applying Lemma \ref{lem:DeltaFactor}
 we have 
$$
0 = \xi_k(\Delta_{2-k})^m f(z)  = \xi_k (\xi_{2-k} \xi_k)^m f(z) = (\xi_k \xi_{2-k})^m (\xi_{k} f(z))= (\Delta_{2-k})^m (\xi_{k} f(z)).
$$
It remains to show that $\xi_k f(z)$ has moderate growth. 
We expand the polyharmonic function $f(z)$ in Fourier series,
using Theorem \ref{lem:abstractFourier},  noting that 
the Fourier series coefficients for $n \ne 0$ only involves $s$-derivatives of $W$-Whittaker functions.
We apply $\xi_k$ term by term to the resulting Fourier series.
By Lemma \ref{le54},  we have  for $n \le -1$ that 
$$
\xi_k\big( y^{-\frac{k}{2}} W_{-\frac{k}{2}, s + \frac{k-1}{2}}(4 \pi |n|y) e^{2 \pi i n x}\big) =
- y^{-(\frac{2-k}{2})} W_{\frac{2-k}{2}, -\bar{s} + \frac{1-k}{2}}(4 \pi |n|y) e^{-2\pi i nx}.
$$
while for $n \ge 1$, 
$$
\xi_k\big( y^{-\frac{k}{2}} W_{\frac{k}{2}, s + \frac{k-1}{2}}(4 \pi ny) e^{2 \pi i n x}\big) =
\overline{s}(1-k - \overline{s}) y^{-(\frac{2-k}{2})} W_{-(\frac{2-k}{2}), -\bar{s} + \frac{1-k}{2}}(4 \pi |n|y) e^{-2\pi i n x}.
$$
Differentiating repeatedly with respect to $s$, and noting that
$$
\frac{\partial}{\partial s}  \xi_k = \xi_k \frac{\partial}{\partial \ol{s}} \quad\mbox{and} \quad\frac{\partial}{\partial \ol{s}} \xi_k = \xi_k \frac{\partial}{\partial s},
$$
 we obtain, for  $n \le -1$, 
$$
\xi_k\Big( \frac{\partial^j}{\partial s^{j}}\big( y^{-\frac{k}{2}} W_{-\frac{k}{2}, s + \frac{k-1}{2}}(4 \pi |n|y) e^{2 \pi i n x}\big) \Big)=
- y^{-(\frac{2-k}{2})} \frac{\partial^j}{\partial \bar{s}^{j}}
W_{\frac{2-k}{2}, -\bar{s} + \frac{1-k}{2}}(4 \pi |n|y) e^{-2\pi i nx}.
$$
For the case $n \ge 1$ the repeated $s$-derivatives give many more terms, but they all involve
polynomials in $\bar{s}$ times $\frac{\partial^k}{\partial \bar{s}^k} W_{\frac{2-k}{2}, -\bar{s} + \frac{1-k}{2}}(4 \pi |n|y)$
with $0 \le k \le j$.   All the resulting $s$-derivatives of the Whittaker W-functions have rapid decay
at the cusp (uniformly for a fixed value $s=s_0$, in our case $s=0$).
This may be shown by differentiating an integral representation of the Whittaker $W$-function with
respect to the $s$-parameter, see \cite{LR15S} for details.
The resulting Fourier series expansion has moderate growth at the cusp, certifying that $\xi_k f(z) \in V_{2-k}^{m}(0)$.
\end{proof}

%
%
\begin{remark}

(1) This argument is  the special case $\lambda=0$ of a result proved in \cite{LR15S}.

(2) There  exist
weight $k$ real-analytic modular forms $f(z)$ {\em not} having moderate growth at the cusp with the property that
$\xi_k f(z)$ has moderate growth at the cusp. The simplest examples are members of $M_k^{!} \backslash M_k$
which have linear exponential growth at the cusp but
 are annihilated by $\xi_k$ since they are holomorphic functions.
\end{remark}

\section{Polyharmonic Non-Liftability of Holomorphic Cusp Forms} \label{sec:newsec6}


\subsection{Weak Maass Forms}\label{sec61}

We consider the liftability problem in the context of larger spaces of weak
Maass forms, in which lifts do exist.  Such spaces were originally introduced in a more general  context
in Bruinier and Funke \cite{BF04}, see also  Bruggeman \cite{Brug14}.

\begin{definition}\label{de61}
For $k \in 2\ZZ$ 
 the space $H_k(0)$ of  weight $k$ {\em Harmonic weak Maass forms}
are those $f: \HH \to \CC$, such that
\begin{enumerate}
\item[(1)] For all $\gamma \in SL(2, \ZZ)$,
$f(\gamma z) = (cz+d)^k f(z).$
\item[(2)] $\Delta_k f(z) = 0$
\item[(3)]{\rm ( Linear exponential growth at the cusp)}
There is a positive constant $A= A(f) < \infty$ such that
$|f(z)| \le e^{Ay}$ for all $y \ge y_0$.
\end{enumerate}
\end{definition}

The vector space $H_k= H_k(0)$ is  infinite dimensional.
By definition the  holomorphic functions in the space 
$H_k$ comprise the space $M_k^{!}$ of {\em weakly holomorphic
 weight $k$ modular forms}, cf.  Definition \ref{de34}.

\begin{proposition}\label{pr62} {\rm (Bruinier and Funke \cite[Prop. 3.2]{BF04})}
For $\lambda=0$ and all $k \in 2\ZZ$ the map $f(z) \mapsto \xi_k f(z)$ defines
a conjugate-linear mapping
$$ \xi_k: H_k \to M_{2-k}^{!}. $$
Its kernel is $M_k^{!}$.
\end{proposition}

This mapping was shown to be surjective by Bruggeman \cite[Theorem 1.1]{Brug14}.
The members of $H_k$  have a Fourier expansion of shape
$$
f(z) =\sum_{n \ll \infty}  c_f^{-}(n) \Gamma(1-k, 4 \pi |n|y) q^n 
+ \sum_{n \gg - \infty} c_f^{+}(n) q^n,
$$
with $q= e^{2 \pi i z}$ and  $z=x+iy$.
Here $n \gg -\infty$ means
that $n$ is bounded below, $n \ll \infty$ means $n$ is bounded above. 
 We call
$$
f^{+}(z) := \sum_{n \gg - \infty} c_f^{+}(n) q^n
$$
the {\em holomorphic part} of $f(z)$, and
$$
f^{-}(z) = \sum_{n \ll \infty}  c_f^{-}(n) \Gamma(k-1, 4 \pi |n|y) q^n
$$
the {\em non-holomorphic part} of $f(z)$. We also call the finite sum
$$
P(f)(z) := \sum_{n \le 0} c_f^{+}(n) q^n
$$
the {\em principal part} of $f(z)$ (or of $f^{+}(z)$.)

The inclusion $V_k^{1}(0) \subset H_k$ immediately follows
by considering the Fourier series of members of $V_k^1(0)$
given in Lemma \ref{lem:FourierExp1Harmonic}.

\begin{definition}\label{de63}
For $k \in 2\ZZ$ the space $H_k^{+}$ of  weight $k$
 {\em harmonic weak Maass forms}\footnote{The terminology is used in
 Bruinier, Ono and Rhoades \cite{BOR08}, who  call this space $H_k$,
 omitting the $+$,
 see \cite[Remark 6]{BOR08}.} 
of eigenvalue $\lambda=0$ are those $f \in H_k^{+}$, whose Fourier expansion
has non-holomorphic part of form
$$
f^{-}(z) =\sum_{n <0}  c_f^{-}(n) \Gamma(k-1, 4 \pi |n|y) q^n,
$$
That is, $f^{-}{z}$ has rapid decay as $y \to \infty$.
\end{definition}

The subspace $H_k^{+}$ has an alternate characterization as the set of all
$f(z) \in H_k$ such that $\xi_k f(z) \in S_k$ is a holomorphic cusp form.

\begin{proposition}\label{pr64} {\rm (Bruinier and Funke \cite[Theorem 1.1]{BF04})}
For $k \in 2\ZZ$ the space $H_k^{+}$ contains $M_k^{!}$ and
there is an exact sequence  (regarded as $\RR$-vector spaces)
$$
0 \rightarrow M_k^{!} \rightarrow H_k^{+} \rightarrow S_{2-k} \rightarrow 0,
$$
in which the third map is $\xi_k$ (which is conjugate-linear).
Moreover, there is a well-defined bilinear form $\{\cdot, \cdot\}$ defined for
$f(z) \in H_k^{+}, g(z) \in S_{2-k}$ by
$$
\{ g, f\} := ( g, \xi_k(f) )_{2-k},
$$
in which $(\cdot, \cdot)_{2-k}$ denotes the Petersson inner  product on modular
forms of weight $2-k$ (at least one a cusp form). This bilinear form gives a
non-degenerate pairing of $S_{2-k}$ with $H_k^{+}\slash M_k^{!}.$
\end{proposition}

The sequence above is  an exact sequence of $\RR$-vector spaces but not an exact sequence of $\CC$-vector spaces
because the map $\xi_k$ is conjugate-linear, compare \cite[Corollary 3.8]{BF04}.

\subsection{Regularized Petersson Inner Product}\label{sec62}

We recall that for weight $k$ holomorphic modular forms the {\em Petersson inner product} is defined on
$M_k \times S_k$ by
$$
(f, g)_k := \int_{\sF} f(z) \overline{g(z)}y^k \frac{dx dy}{y^2},
$$
in which $\sF=\{ z : \, |z| \ge 1, 0 \le x \le 1\}$ is the standard fundamental domain
for $SL(2, \ZZ)$.

 A {\em (regularized) inner product} generalizing the Petersson inner product
 was introduced in Borcherds \cite{Bor99}. Let $\sF_T$ denote
 the standard fundamental domain for $SL(2, \ZZ)$ cut off at height $T$,
 $$
 \sF_T = \{ z: \, , |z| \ge 1 \, \mbox{with} \,\, 0 \le x \le 1, \, y\le T\}.
 $$
 Let $g\in M_k$ 
and let $h\in M_k^!$ be a weight $k$ weakly holomorphic modular form of weight $k$ on $\SL_2(\Z)$.
 Then for modular forms of weight $k$, set 
$$
(g, h)^{reg} := \lim_{T\to \infty} \int_{\cF_T} g(z) \ol{h(z)} y^k \frac{dxdy}{y^2}.
$$
\begin{remark}
If $h$ is a cusp form, then $(g, h)^{reg}$ is the usual Petersson inner product.
\end{remark}

%
%

\begin{lemma}\label{lem:regularizedInnerProduct}
Let $f(z)\in H_{2-k}^{+}$ have
 $\xi_{2-k} f(z)  = g(z)$, where $g \in S_{k}$ is  a nonzero weight $k$ holomorphic cusp form.
Then 
$$
f(z) = \sum_{n\gg -\infty} c_{f}^+(n) q^n + \sum_{n <0} c_{f}^-(n) \Gamma(k-1, 4\pi \abs{n} y) q^n,
$$
with $q=e^{2 \pi i z}$,
and 
\begin{equation}\label{eq-claim1}
g(z) = \xi_{2-k} f(z) = -  \sum_{n=1}^\infty \ol{c_f^{-}(-n)}(4 \pi n)^{k-1} q^n.
\end{equation}
Moreover, if $h(z)  \in M_{k}$ with
$h(z) = \sum_{n=0}^{\infty} c_h(n)^{+} q^n$
then the Petersson inner product
\begin{equation}\label{IPF}
(h, \xi_{2-k} f)_{k} = \sum_{n=0}^{\infty} c_h^{+}(n) c_f^{+}(-n).
\end{equation}

In particular, for $h = g$, then 
\begin{equation}\label{eq-claim2}
(g, g)_{k} = (g, \xi_{2-k} f)_{k}  
=  -\sum_{n=1}^{\infty}  (4 \pi n)^{k-1} \overline{c_{f}^{-}(-n)} \,c_f^{+}(-n).
\end{equation}
\end{lemma}

\begin{proof}
The Fourier series \eqref{eq-claim1} follows from $\xi_k$ applied
term by term to the Fourier series of $f(z)$, using Lemma \ref{lem:FourierExp1Harmonic}
and Lemma \ref{le54} (1).
(See also the introduction to \cite{BOR08}.)
The conversion  \eqref{fourier-convert} is applied twice, as 
 $$
 b_{-n}(4 \pi |n|)^{\frac{k}{2}-1} = \tilde{b}_{-n}\quad \longrightarrow_{\xi_{2-k}} \quad
  -\overline{\tilde{b}_{-n}}= \tilde{a}_n = a_n(4 \pi n)^{-\frac{k}{2}},
 $$
 yielding  $a_n = -\overline{b_{-n}} (4 \pi n)^{k-1} $
The  Petersson inner product formula \eqref{IPF}
is a special case of Bruinier-Funke \cite[Prop. 3.4]{BF04}.
\end{proof}

\subsection{Non-Liftability of Holomorphic Cusp Forms}\label{sec63}

We show that there are no preimages under $\xi_{2-k}$ of 
holomorphic cusp forms for $PSL(2, \ZZ)$
that have polynomial growth at the cusp. 

%
%

\begin{proposition}\label{prop:liftOfCuspForms}
Let $g \in S_k$ be a weight $k$ holomorphic cusp form for $\SL_2(\Z)$. 

(1) Let $M$ be a non-holomorphic weight $k$ modular form for $\SL_2(\Z)$ satisfying 
$$
\Delta_k M(z) = g(z).
$$
Then $M$ cannot have polynomial growth in $y$ as $y\to \infty$. \\

(2) There is no weight $2-k$  modular form $f(z) \in V_{2-k}^1(0)$ such that 
$$
\xi_{2-k}f(z) = g(z).
$$
\end{proposition}
\begin{proof}

(1) Assume for a contradiction that such an $M(z)$ exists, in which case  $M(z)\in V_k^2(0)$.
Lemma \ref{lem:subspaceXi} now shows that $f (z) := \xi_k M(z) \in V_{2-k}^2(0)$.
We now assert $f(z) \in V_{2-k}^1(0)$. This holds since
$f(z)$ has moderate growth  and
$$\xi_{2-k} f(z) = \xi_{2-k}(\xi_k M)(z)= \Delta_k(M)(z) = g(z),$$
 which yields
$$
\Delta_{2-k}(f)(z) = \xi_{k}(\xi_{2-k}f)(z) = \xi_kg (z) =0.
$$
because $g(z)$ is a holomorphic form. Thus (1) will follow if we prove (2).
 
(2) We have $f(z) \in V_{2-k}^1(0)$ and we 
  $f(z) \in H_{2-k}$ since $V_{2-k}^1(0) \subset H_{2-k}$.
Since
$\xi_{2-k} f\in S_k$ we  have $f(z) \in H_{2-k}^{+}$.
According to Proposition \ref{pr64} there is a unique class of
lifts $[f] \in H_{2-k}^{+} \slash M_{2-k}^{!}$ whose image is $g(z)$.
We need to prove  that no member $f(z)$ of this class has moderate growth
at the cusp.
Since  our given $f(z) \in V_{2-k}^1(0)$, 
by Lemma \ref{lem:FourierExp1Harmonic} it has 
 a Fourier expansion of the form 
$$
f(z) =  \sum_{n=1}^{\infty} b_{-n} \Gamma(k-1, 4\pi \abs{n} y) e^{-2\pi i n z}
+  ( b_0 y^{1-k}+a_0)+\sum_{n =1}^{\infty} a_n e^{2\pi i n z} .
$$ 
By Lemma \ref{lem:regularizedInnerProduct} we have 
$$
g(z) = \xi_{2-k}f(z) = 
(k-1) \overline{b_0} - \sum_{n = 1}^{\infty} \overline{b_{-n}}(4 \pi n)^{k-1}e^{2 \pi i nz}.
$$
 We see in addition that $b_0 =0$, since $g(z)$
 is a cusp form. 
   We next compute the Petersson
 inner product $(g,g)_k$ using 
 Lemma \ref{lem:regularizedInnerProduct} to be
$$
(g,g)_k = (g, \xi_{2-k}f)_k = - \sum_{n \ge 1} b_{-n}(4\pi n)^{k-1} a_{-n}=0.
$$
 Here  we used $c_f^{+} (-n) = a_{-n}=0$ for all
$n \ge 1$, because the  Fourier terms  $c_f^{+}(-n) q^{-n}$ are not of moderate growth.
This is a contradiction  because
the Petersson inner product $(g,g)_k > 0$ since
$g(z) \ne 0$.  We conclude such an $f(z)$ cannot exist, which proves (2).
\end{proof}

\begin{remark}
(1) One  may explicitly construct  harmonic weak Maass forms  that do map to holomorphic
cusp forms under $\xi_{2-k}$ using  Maass-Poincar\'{e} series, see 
Bringmann and Ono \cite{BrO07},
Bruinier et al.  \cite{BOR08}, and the discussion in Kent \cite[Sec 1.3.2]{Kent10}.

(2) It is  interesting to check how the 
non-holomorphic Eisenstein series $f(z) = \htE_{2-k}(z, 0)$ evades the 
contradiction in the proof above. To allow Eisenstein series we must  use
a regularized inner product. 
A nonzero value is then permitted in the inner product calculation
due to the presence of the $n=0$ term
in the inner product $(g, g)^{reg}.$ 
\end{remark}


\section{Action of $\xi_k$ on non-holomorphic Eisenstein series}\label{sec:7}

For later use we  compute the action of the $\xi_k$-operator on 
the non-holomorphic Eisenstein series.

%
%

\begin{proposition}\label{pr71}
Let $k \in 2 \ZZ$. Then
$$
\xi_k \htE_k(z, s) = \begin{cases}
\htE_{2-k}(z, -\overline{s}) & \mbox{if} \quad k \le 0,\\
\overline{s}(\overline{s} + k -1 ) \htE_{2-k}(z, -\overline{s}) & \mbox{if} \quad k \ge 2.
\end{cases}
$$
In addition
$$
\xi_k \dhtE_k(z, s) = \begin{cases}
\dhtE_{2-k}(z, -\overline{s}) & \mbox{if} \quad k \le 0,\\
\overline{s}(\overline{s} + k -1 ) \dhtE_{2-k}(z, -\overline{s}) & \mbox{if} \quad k \ge 2.
\end{cases}
$$
\end{proposition}

We obtain this result by first studying the action of $\xi_k$ on
the uncompleted Eistenstein series $E_k(z, s).$

%
%

\begin{lemma}\label{lem:xiDelta_arbWeight}
For  $k \in 2\ZZ$ there holds
$$\xi_{k} E_k(z, s) = \overline{s} E_{2-k} (z, \overline{s}+ k  - 1).$$
\end{lemma}

\begin{proof}
Using Lemma \ref{lem:xi_commute} we have, taking $\psi_s(z) = y^s$, 
we have $\psi_s(\gamma z) = \frac{y^s}{|cz+d|^{2s}}$. Then, for $Re(s)> 2$,
and with $\Gamma_{\infty} =\{[ \begin{smallmatrix}1 & n \\ 0 & 1\end{smallmatrix}]: \, n \in \ZZ\}$,
we have
\begin{align*}
\xi_k E_k(z, s) =& \frac{1}{2} \sum_{\gamma =[ \begin{smallmatrix}* & * \\ c & d\end{smallmatrix}]
\in \Gamma_{\infty} \backslash SL(2, \ZZ)}
 \xi_k \( (cz+d)^{-k} \psi_s (\gamma z)\)\\
  =& \frac{1}{2} \sum_{\gamma \in \Gamma_{\infty}\backslash SL(2, \ZZ)}
   (cz+d)^{k-2} \( \xi_k \psi_s\) (\gamma z) \\
  =& \frac{1}{2}  \sum_{\gamma \in \Gamma_{\infty}\backslash SL(2, \ZZ)}(cz+d)^{k-2} \overline{s} \psi_{\overline{s} -1 + k} (\gamma z) \\
  =& \overline{s} E_{2-k} (z, \overline{s} - 1 + k)
\end{align*} 
using the identity
$2 i y^k \overline{ \frac{\partial }{\partial \overline{z}} y^s} = \overline{s} y^{\overline{s} - 1 + k}$.
The identity holds  on compact subsets of  $Re(s)>2$ and  of  $z \in \HH$
because the sums on both sides converge absolutely and 
uniformly  on these domains.
It now follows for
all $s \in \CC$ and all $z \in \HH$ by analytic continuation, away from poles.
\end{proof}

One can easily deduce from Lemma \ref{lem:xiDelta_arbWeight} a
direct proof of   Theorem \ref{th38}(3). Namely 
since  $\Delta_k = \xi_{2-k}\circ \xi_k$  we obtain 
$$\Delta_k E_k (z, s) = \xi_{2-k} \overline{s} E_{2-k} (z, \overline{s} -1 + k) = s(s+k-1) E_k(z,s).$$
%

\begin{proof}[Proof of Proposition \ref{pr71}.]
Lemma \ref{lem:xiDelta_arbWeight} gives
$$
\xi_k \htE_k(z,s) = \ol{s} \frac{\Gamma \( s+ \frac{k}{2} + \frac{\abs{k}}{2})\)}{\Gamma\( \ol{s} + \frac{k}{2} + \frac{\abs{2-k}}{2}\)}\htE_{2-k}(z, \ol{s} - 1+k)  
= \ol{s} \frac{\Gamma \( s+ \frac{k}{2} + \frac{\abs{k}}{2}) \)}{\Gamma\( \ol{s} + \frac{k}{2} + \frac{\abs{2-k}}{2}\)}\htE_{2-k}(z, -\ol{s} ),
$$
where the functional equation of $\htE_k(z, s)$ was used to get the rightmost equality.
For $k\ge 2$, the $\Gamma$-quotient is $\frac{\Gamma(\ol{s}+k)}{\Gamma(\ol{s}+k-1)} = (\ol{s}+k-1)$.
For $k \le 0$, the $\Gamma$-quotient is $\frac{\Gamma(\ol{s})}{\Gamma(\ol{s}+1)} = \frac{1}{\ol{s}}$.  The result for $\htE_k(z, s)$ follows. 

 The result for $\dhtE(z, s) = (s+ \frac{k}{2}) (s+ \frac{k}{2}-1) \htE_k(z, s)$
follows on noting that
$$
\overline{ (s+ \frac{k}{2}) (s+ \frac{k}{2}-1) }= 
(-\overline{s}+ \frac{2-k}{2}) (-\overline{s} +\frac{2-k}{2}-1).
 $$
\end{proof}

\section{Taylor Series of  Non-holomorphic Eisenstein Series}\label{sec:TS}

We  show that the Taylor series coefficients of the
doubly completed Eisenstein series $\dhtE_k(z, s)$
in the $s$-variable  at $s=0$ define polyharmonic Maass forms
of eigenvalue $\lambda=s(s+k-1)$.
We consider  the doubly-completed Eisenstein series $\dhtE_k(z, s)$ 
rather than the singly completed $\htE_k(z, s)$ because it
is an entire function of $s$ for all $k$. 
The case $\lambda=0$ has special features compared
to the general $\lambda$ case, which we treat in \cite{LR15S};

We will use 
rescaled Taylor coefficients of $E_k(z, 0)$, matching the
notation used in Section  \ref{sec:Weight}, namely
\begin{equation}\label{eqn:E_k-coefficients}
\dhtE_k(z,s) = 
\begin{cases} 
\sum_{n=0}^\infty F_{n, k} (z) s^n  & \mbox{for weights} \,\,  k \le 0, \\ 
 \sum_{n=0}^\infty G_{n,k}(z) s^n & \mbox{for weights} \,\,  k \ge 2.
 \end{cases}
 \end{equation}
In consequence the derivatives with respect to $s$ exhibit
\begin{equation}\label{eqn:F-coefficients}
\frac{\partial^n}{\partial s^n} \dhtE_k(z, s) |_{s=0}
= 
\begin{cases} 
 n! F_{n, k} (z)   & \mbox{for weights} \,\,  k \le 0,  \\ 
 n! G_{n,k}(z) & \mbox{for weights} \,\,  k \ge 2.
 \end{cases}
\end{equation}

%
%
%
%

\subsection{Initial Taylor coefficients at $s=0$  of 
doubly completed  weight $0$ and $2$ Eisenstein series.}\label{sec:new81}

In what follows we  use the abbreviated notation $F_n(z)= F_{n, 0}(z)$
and $G_n(z) = G_{n, 2}(z)$.

%
%

\begin{lemma}\label{le27}
The Taylor series  of the doubly completed Eisenstein series
 $\dhtE_{0}(z, s) = \sum_{n=0}^{\infty} F_{n}(z) s^n$ 
 at $s_0=0$ has
$$F_{0}(z) =\dhtE_0(z, 0) = \frac{1}{2}.$$
In addition
\begin{align*}
F_{1}(z) = &\frac{\partial}{\partial s}\dhtE_0(z,s) \big|_{s=0}  =
- \frac{1}{2}\gamma+ \log(4\pi) +
\log \( \sqrt{y} \abs{\Delta(z)}^{\frac{1}{12}}\)
\end{align*}
\end{lemma} 

\begin{proof}
The first assertion is derived from the Fourier expansion
 in Proposition \ref{prop:FourierExpansionArbitrary}, 
using $\dhtE_0(z, s)= s(s-1) \htE_0(z, s)$,
after observing that for $k=0$ all the Fourier terms with $n \ne 0$ vanish,
and  the constant term becomes
$$
s(s-1) C_0(y, s)|_{s=0} = s(s-1) \htz(2s) y^s)\big|_{s=0} + s(s-1)\zeta(2-2s)y^{1-s}\big|_{s=0} = \frac{1}{2},
$$
coming from the simple pole of $\htz(2s)$ at $s=0$ having residue $-\frac{1}{2}.$

The second 
assertion is deduced from  Kronecker's first limit formula \cite[Chap. 1, Theorem 1]{Sie80}
which states\footnote{Siegel's definition of the Eisenstein series has an extra factor of $2$
compared to \eqref{100}.}
$$
2E_0(z, s) = \frac{\pi}{s-1} +2 \pi (\gamma - \log 2) - \log(\sqrt{y}\abs{\eta(z)}^{2}) +O( (s-1)),
$$
where $\gamma$ is Euler's constant, and
the Dedekind eta function satisfies  $|\eta(z)| = |\Delta(z)|^{1/24}$.
Using the functional equation for $\htE_{0}(z, s)= \htE_0(z, 1-s)$ we  obtain
$$
E_0(z, s) = \gamma(s) E_0(z, 1-s)
$$
with
$$
\gamma(s) := \frac{ \pi^{-(1-s)}\Gamma(1-s)}{\pi^{-s} \Gamma(s)}
= s \pi^{2s- 1 } \frac{\Gamma(1-s)}{ \Gamma(1+s)}.
$$
Computing Laurent expansions at $s=0$ of $\gamma(s)$ and $E_0(z, 1-s)$
yields  analogous Kronecker's limit formula to be derived at $s=0$.
We obtain
$$
E_0(z, s) = -\frac{1}{2} + (-\log (2 \pi) - \log \sqrt{y} |\Delta(z)|^{\frac{1}{12}}| )s +O(s^2).
$$
It was given in Stark \cite{Sta77}, see also Katayama \cite[(3.2.8)]{Kat10}. 
Another Laurent series calculation using $\dhtE_0(z,s)= s(s-1) \htE_0(z, s)$ yields
$$
\dhtE_0(z, s) = \frac{1}{2} - \big( \frac{1}{2} \gamma +\log(4 \pi) 
+ \log (\sqrt{y}|\Delta(z)|^{\frac{1}{12}})\big)s +O(s^2).
$$
%
\end{proof}

%
%

\begin{lemma}\label{le37} 
The Taylor series  of the doubly completed Eisenstein series
$\dhtE_{2}(z, s) = \sum_{n=0}^{\infty} G_{n}(z) s^n$ at $s_0=0$ has
$$
G_{0}(z) \equiv 0
$$ 
and 
$$
G_{1}(z) =\frac{\partial}{\partial s}|_{s=0} \dhtE_2(z, 0) 
= -\frac{\pi}{6} 
- \frac{1}{2y} + 4 \pi \big( \sum_{n=1}^{\infty} \sigma_1(n)e^{2 \pi i n z}\big)
$$
where $z= x+iy \in \HH$. This function is holomorphic in $z$ after excluding 
the non-holomorphic term $\frac{1}{2y}$ appearing in the constant term of its Fourier expansion. 
\end{lemma} 

\begin{proof}
We start from the Fourier series expansion of $\htE_2(z, s)$ given in
Proposition \ref{prop:FourierExpansionArbitrary}. It gives
$$
\htE_2(z, 0)= ( \hat{\zeta}(2) + \frac{1}{2y}) -  4 \pi \big(\sum_{n=1}^{\infty} \sigma_1(n)e^{2 \pi i n z}\big),
$$
where we use 
$W_{1, \frac{1}{2}}(y) = y \, e^{-\frac{1}{2}y}$ (\cite[(13.18.2)]{NIST}) to get
$\frac{1}{4 \pi n y}W_{1, \frac{1}{2}}(4 \pi ny) e^{2 \pi i n x} = e^{2 \pi i n z},$
and  we also use $\lim_{s \to 0} (-s) \hat{\zeta}(-2s) y^s =  \frac{1}{2}$.
It is analytic at $s=0$ so that $\dhtE_2(z, s)=s(s+1) \htE(z, s)$ has
$G_0(z) = \dhtE_2(z, 0) \equiv 0.$
Furthermore we have 
$G_{1}(z) = \frac{\partial}{\partial s} \dhtE_2(z, s) \big|_{s=0} = -\htE_2(z, 0),$
which with $\hat{\zeta}(2)= \frac{\pi}{6}$ gives the result.
\end{proof}


%
%
%
%

\subsection{Recursions for $\Delta_k$ and $\xi_k$ action on Taylor  coefficients of Eisenstein series}\label{sec:new72}

We establish  recursion relations under $\Delta_k$ and $\xi_k$
relating the Taylor coefficients $F_k(z)$ and $G_k(z)$.

\begin{proposition}\label{lem:xiDelta_arbWeight_Taylor}
With the Taylor coefficients of $E_k(z,s)$ defined in \eqref{eqn:E_k-coefficients} 
and setting $F_{n,k}^{-}(z) = (-1)^n F_{n,k}(z)$, we have 
\begin{enumerate}
\item For $k\ge2$ an even integer, 
\begin{equation}\label{G-recursion}
\Delta_k G_{n,k} (z) =  (k-1) G_{n-1, k}(z) + G_{n-2, k}(z),
\end{equation}
where we define $G_{-1, k}(z) = G_{-2, k}(z) \equiv 0$.
\item For $k \le 0$ an even integer, 
\begin{equation}\label{F-recursion}
\Delta_{k} F_{n,k}^{-}(z)  = (1-k) F_{n-1, k}^{-} (z) + F_{n-2, k}^{-}(z),
\end{equation}
where we define $F_{-1, k}^{-}(z) = F_{-2, k}^{-}(z) \equiv 0$.
\item For $k \ge 2$ an even integer, 
\begin{equation}\label{G-to-F-recursion}
\xi_{k} G_{n,k} (z) =  
(k-1) F_{n-1, 2-k}^{-}(z) + F_{n-2, 2-k}^{-}(z).
\end{equation}
\item For $k \le 0$ an even integer, 
\begin{equation}\label{F-to-G-recursion}
\xi_{k} F_{n,k}^{-}(z) =  G_{n,2-k}(z).
\end{equation}
\end{enumerate}
\end{proposition}

\begin{proof}
(1) Let  $k \ge 2$. Then
$$
\Delta_k ( \dhtE_k(z, s)) = \Delta_k\Big( \sum_{n=0}^{\infty} G_{n,k}(z) \, s^n \Big) = \sum_{n=0}^{\infty} \Delta_k G_{n,k}(z) s^n.
$$
By Theorem \ref{th38} (3), 
\begin{eqnarray*}
\Delta_k ( \dhtE_k(z, s)) &=& s(s+k-1) \dhtE_k(z, s)\\
&=& \sum_{n=0}^{\infty} G_{n,k} s^{n+2} + (k-1) \sum_{n=0}^{\infty} G_{n, k} s^{n+1}\\
&=& (k-1)G_{0,k}(z) s + \sum_{n=2}^{\infty}  \big( (k-1)G_{n-1, k}(z) + G_{n-2, k}(z)\big) s^n
\end{eqnarray*}
Term by term comparison establishes 
$\Delta_k G_{n,k}(z) = (k-1)G_{n-1, k}(z) + G_{n-2, k}(z)$
for all $n \ge 0$,
using the convention that $G_{-1,k}(z) = G_{-2, k}(z) \equiv 0.$

(2) Let $k \ge 0$. An argument  similar to (1) yields 
$\Delta_k F_{n,k}(z) = (k-1) F_{n-1,k}(z) + F_{n-2, k}(z)$
for all $n \ge 0$, 
using the convention that $F_{-1,k}(z) = F_{-2, k}(z) \equiv 0.$
Substituting $F_{n,k}^{-}(z) = (-1)^n F_{n,k}(z) $ yields \eqref{F-recursion}.

(3) We argue similarly to (1) using Proposition \ref{pr71}.
Suppose $k \ge 2$. Then
$$
\xi_k ( \dhtE_k(z, s)) = \xi_k\Big( \sum_{n=0}^{\infty} G_{n,k}(z) \, s^n \Big) = \sum_{n=0}^{\infty} \xi_k G_{n,k}(z) (\ol{s}^n).
$$
By Proposition \ref{pr71},
\begin{eqnarray*}
\xi_k ( \dhtE_k(z, s)) &=& \ol{s}(\ol{s}+k-1) \dhtE_{2-k}(z, -\ol{s})
=  \ol{s}(\ol{s}+k-1) \Big( \sum_{n=0}^{\infty} F_{n,2-k}^{-}(z) (\ol{s})^n\Big) \\
&=& (k-1) F_{0, 2-k}(z) \ol{s} + \sum_{n=2}^{\infty} \big( (k-1) F_{n-1, 2-k}^{-}(z) + F_{n-2, 2-k}^{-}(z)\big) (\ol{s})^n.
\end{eqnarray*}
Term by term comparison in $\ol{s}^n$ yields, 
$\xi_k G_{n,k}(z) = (k-1) F_{n-1, 2-k}^{-}(z) + F_{n-2, 2-k}^{-}(z).$
for all $n \ge 0$.

(4) Let $k \le 0$, so
$$
\xi_k ( \dhtE_k(z, s)) =  \sum_{n=0}^{\infty} \xi_k F_{n,k}(z) (\ol{s}^n) = \sum_{n=0}^{\infty} F_{n,k}^{-}(z) (-\ol{s}^n).
$$
By Proposition \ref{pr71},
$$
\xi_k ( \dhtE_k(z, s)) =  \dhtE_{2-k}(z, -\ol{s})= \ol{s}(\ol{s}+k-1) \sum_{n=0}^{\infty} G_{n,2-k}(z) (-\ol{s})^n.
$$
Term by term comparison in $(-\ol{s})^n$ yields 
$\xi_k F_{n,k}^{-}(z) = G_{n, k}(z)$ for all $n \ge0$.
\end{proof}

%
%
%
%

\subsection{Taylor  coefficients of Eisenstein series at $s=0$ are polyharmonic Maass forms}\label{sec:new83}

Now we show  that  these Taylor coefficients are polyharmonic Maass forms.

%
%
%
%
\begin{theorem}\label{thm:new84}
With the Taylor coefficients of $E_k(z,s)$ defined in \eqref{eqn:E_k-coefficients} 
we have:
\begin{enumerate}
\item
For all $n \ge 0$  the  functions $F_{n,k}(z)$ 
(for weights $k \le 0$) and $G_{n,k}(z)$ (for weights $k \ge 2$) are polyharmonic Maass forms
of depth at most $n+1$, i.e. they belong to $V_k^{n+1}(0)$. 
\item
For weights $k \le 0$ and $n \ge 0$ the function $F_{n,k} (z) \in V_{k}^{n+1}(0) \setminus V_{k}^{n}(0)$.
For weights  $k \ge 4$ and $n \ge 0$ the function $G_{n,k}(z) \in V_{k}^{n+1}(0) \setminus V_{k}^n(0)$.
For weight $k=2$ and $n \ge 1$ the function
$G_{n,2}(z) \in V_{2}^{n}(0) \setminus V_{k}^{n-1}(0)$.
\end{enumerate}
\end{theorem}

\begin{proof} 
(1) The $(n+1)$-harmonicity of both $F_{n,k}(z)$ ($k \le 0$) and $G_{n,k}(z)$ ( $k \ge 2$)
follows by induction on $n$ using the recursions (1) and (2) of
Proposition \ref{lem:xiDelta_arbWeight_Taylor}. These recursions
establish the base cases $\Delta_k F_{0,k}(z) =0$, $\Delta_k G_{0,k}(z) =0$ and
$\Delta_k  F_{1,k}(z) = (k-1)F_{0,k}(z)$, $\Delta_k G_{1, k}(z) = (k-1) G_{0,k}(z)$.
The induction step for $n \ge 2$ uses the two-step recursions \eqref{G-recursion} and \eqref{F-recursion}.

The functions $F_{n,k}(z)$ ( resp. $G_{n,k}(z) $) transform as weight $k$ modular forms
as a property inherited from $\dhtE_k(z, s)$.
To show membership of these functions in $V_{k}^{n+1}(0)$ it remains to
show these functions have moderate growth at the cusp. 
We indicate details for $F_{n,k}(z)$, the argument for $G_{n,k}(z)$ being similar. We 
use the $s$-derivative relation 
$$
F_{m,k}(z) = \frac{1}{m!} \frac{\partial^m}{\partial s^m} \dhtE(z, s) \Big|_{s=0}.
$$
We apply it term by term to the Fourier series for $\dhtE(z, s)$ and then set $s=0$ to obtain the Fourier series for $F_{m, k}(z)$.
Applying $\frac{\partial }{\partial s}$ repeatedly to the Fourier coefficients of $\dhtE(z, s)$
reveals
that these coefficients 
involve polynomials in $s$ of degree at most $2m$
times terms of the form $y^{-\frac{k}{2}}\frac{\partial^j}{\partial s^j} W_{\frac{1}{2}\epsilon k, s+ \frac{k-1}{2}}( 4 \pi |n|y) e^{2 \pi i n x}$
times  for $0 \le j \le m-1$ given in Theorem \ref{th41}. 
When we set $s=0$ it  follows that the individual nonconstant Fourier  coefficients $(n \ne 0)$
each  have rapid decay at the cusp, while the constant term has at most polynomial growth in $y$ as $y \to \infty$,
bounded by $y^{k+\epsilon}$.
 Finally uniform decay estimates in $n$ of the $W$-Whittaker function family and 
a bounded number  $s$-derivatives  of them implies  that the full Fourier series of $F_{m, k}(z)$ has moderate growth at the cusp,
compare \cite{LR15S}.

(2) For $k \ne 2$ prove the result on membership in $V_k^{n+1}(0) \smallsetminus V_k^n(0)$ by induction on $n \ge 0$.
By convention  $V_k^{0}(0) = \{0\}$,
and the base case asserts $F_{0, k}(z) \not\equiv 0$ for weight $k \ge 0$ and $G_{0, k}(z) \not\equiv 0$ for weight $k \ge 4$.
The  non-vanishing of
$F_{0, k}(z)$ for weights $k \ge -2$ and of $G_{0, k}(z)$ for weight $k \ge 4$ are equivalent to non-vanishing
of $\dhtE_k(z, 0)$, which follows from non-vanishing of the constant term in
the Fourier series expansion of $\htE(z, 0)$ in Proposition \ref{prop:FourierExpansionArbitrary}, 
and the fact
that $(s+ \frac{k}{2}) (1-(s + \frac{k}{2})|_{s=0} \ne 0$.  The non-vanishing for weight $k=0$ is 
given by Lemma \ref{le27}. For the step $n=1$ we have for $k \le 0$ that
$\Delta_k F_{1, k}(z) = (k-1) F_{0, k}(z) \in V_k^{1}(0) \smallsetminus V_k^0(0)$, which certifies
$F_{1, k}(z) \in V_k^{2}(0) \smallsetminus V_k^{1}(0)$. Similarly for $k \ge 4$ we have 
$\Delta_k G_{1, k}(z) = (k-1) G_{0, k}(z) \in V_k^{1}(0) \smallsetminus V_k^0(0)$, which certifies
$G_{1, k}(z) \in V_k^{2}(0) \smallsetminus V_k^{1}(0)$.
The induction step with $n \ge 2$  is completed for $k \le 0$ using the recursion \eqref{F-recursion}
using $(k-1)F_{n-1, k}(z) \in V_{k}^n(0) \smallsetminus V_{k}^{n-1}(0)$ and $F_{n-2, k}(z) \in V_k^{n-1}(0)$
yielding $(k-1)F_{n-1, k}(z) + F_{n-2, k}(z)  \in V_{k}^n(0) \smallsetminus V_{k}^{n-1}(0)$,
certifying $F_{n,k}(z) \in V_{k}^{n+1}(0) \smallsetminus V_{k}^{n}(0).$ The argument for $k \ge 4$
is similar.

Weight $k=2$ is exceptional, because we have $G_{0,2}(z) \equiv 0$ by Lemma \ref{le37}.
In this case Lemma \ref{le37} gives $G_{1, 2}(z) \not \equiv 0$, 
and we prove $G_{n,k}(z) \in V_{k}^n(0) \smallsetminus V_{k}^{n-1}(0)$
 for $n \ge 1$ by a similar induction, using $G_{1, 2}(z) \not \equiv 0$ as the base case.
\end{proof}

%
%
%
%

\subsection{Polyharmonic Eisenstein  space}  \label{sec:new84}

Theorem \ref{thm:new84} motivates the following definition,
which is relevant to Theorem \ref{thm:arbitraryweight}.

\begin{definition}
 The {\em polyharmonic Eisenstein space} $E_k^{m}(0)$ is the vector space of weight $k$ polyharmonic
Maass forms of harmonic depth at most $m$ that are spanned by the Taylor series coefficients of
the doubly-completed Eisenstein series $\dhtE_k(z, s)$. 
\end{definition} 


\begin{corollary}\label{cor86}
For each $k \in 2\ZZ$ and each $m \ge 1$ the polyharmonic Eisenstein space $E_k^{m}(0)$
has dimension $m$ and is contained in $V_k^m(0)$. 
For $k \le 0$ a basis   is $\{F_{j,k}(z): \, 0 \le j \le m-1\}.$
For $k=2$ a basis is $\{ G_{j, 2}(z): 1 \le j \le m\}.$
For $k \ge 4$ a basis is $\{G_{j, k}(z); \, 0 \le j \le m-1\}$.
\end{corollary}
\begin{proof} 
Theorem \ref{thm:new84} identifies the  given Taylor coefficient
functions  as being
in the space $E_{k}^m(0)$ and being linearly independent.
At the same time it identifies
each  Taylor coefficient functions as being linearly independent
of the full set of  lower numbered ones. 
\end{proof}

\section{Modified Taylor Coefficient Basis}\label{sec:new9}

We  obtain modified basis functions  of the polyharmonic Eisenstein space $E_k^{m}(0)$
that simplifies the action of $\Delta_k$ and $\xi_k$,
via a  triangular  change of basis  of the Taylor coefficients of $\dhtE_k(z, s)$.

\begin{theorem}\label{thm:76}
{\rm (Modified polyharmonic Eisenstein space basis)}
Let $k \in 2\ZZ$, with  $k \ge 2$.
Let  
$\{ c_{n, k, \ell}: \, n \ge 0; \,1 \le \ell \le n+1\}$ 
be a set of (complex) constants, 
and define for $n \ge 0$ the  set of functions
\begin{eqnarray}
\wt{G}_{n,k} (z) & := & G_{n,k}(z) +\sum_{\ell=1}^n (-1)^{\ell} c_{n, k, \ell}\,G_{n-\ell, k}(z) \label{G-recursion2},\\
\wt{F}_{n, 2-k} (z) & := &(-1)^n \Big( F_{n,k}(z) + \sum_{\ell=1}^{n} \ol{c}_{n, k, \ell} \, F_{n-\ell, 2-k}(z) \Big) \label{F-recursion2},
\end{eqnarray}
in which  $G_{n,k}(z)$ are the Taylor coefficients of $\dhtE_k(z, s)$ and
$F_{n,k}(z)$ are the Taylor coefficients of $\dhtE_{2-k}(z, s)$.
Then for all sets of  constants that obey for all $n \ge 1$ the connection equations:
\begin{equation}\label{rec-zero}
 c_{n, k, \ell} = \frac{1}{k-1} c_{n, k, \ell-1}
+ c_{n-1,k,\ell} , \quad \mbox{for} \quad 0 \le \ell \le n,
\end{equation}
in which we set  $c_{n,k,0}=1$ and $c_{n,k,-1}=0$,
and which involve the free constants $c_{n, k, n+1}$ (for $n \ge 0$) 
these functions will
will simultaneously satisfy:
\begin{enumerate}
\item[(1)] {\rm (``Ramp"  relations)}
\begin{eqnarray}\label{G-to-F2}
\xi_{k} \,\wt{G}_{n, k}(z) \quad &= &( k-1) \,\wt{F}_{n-1, 2-k}(z)    \quad\mbox{for} \quad n \ge 1, \\
\label{F-to-G2}
\xi_{2-k} \wt{F}_{n,2-k}(z) \,&=& \quad \quad \quad \wt{G}_{n, k}(z) \quad \quad\mbox{for} \quad n \ge 1.
\end{eqnarray}
\item[(2)] {\rm (``Tower" relations)}.
\begin{eqnarray}\label{G-to-G2}
\Delta_{k}\ \wt{G}_{n, 2-k}(z) &= &(k-1)\wt{G}_{n-1,k}(z) \quad \quad \mbox{for} \quad n \ge 1,\\
\label{F-to-F2}
\Delta_{2-k} \wt{F}_{n,2-k}(z) & = &(k-1) \wt{F}_{n-1,2-k}(z) \quad \mbox{for} \quad n \ge 1.
\end{eqnarray}
\end{enumerate}
\end{theorem}

\begin{remark}
The connection equations  have infinitely many solutions, since may
 specify arbitrarily the values of $c_{n, k, n+1}$ for all $n \ge 0$.
\end{remark}

\begin{proof}
Suppose $k \ge 2$.
It suffices to verify the ``ramp" relations (1) hold, since the
``tower relations" (2) then follow using $\Delta_k = \xi_k \xi_{2-k}$
and $\Delta{2-k}= \xi_{2-k} \xi_k$. 
Set $c_{n, k, 0}=1$ and $c_{n, k, -1} =0$ for all $n  \ge 0$.
By Proposition \ref{lem:xiDelta_arbWeight_Taylor} (3) we obtain
\begin{eqnarray*}
\xi_k \wt{G}_{n,k}(z) &=& \xi_k \Big( \sum_{\ell =0}^n (-1)^{\ell} c_{n,k,\ell} G_{n-\ell, k}(z)\Big)\\
&=& \sum_{\ell=0}^n (-1)^{\ell} \ol{c}_{n, k, \ell} \Big( (k-1)(-1)^{n -\ell -1} F_{n-\ell-1, 2-k}(z) + (-1)^{n - \ell -2} F_{n-\ell -2, 2-k}(z) \Big)\\
&=& (-1)^{n-1} \sum_{\ell=0}^n \ol{c}_{n, k,\ell} \Big( (k-1) F_{n -\ell -1, 2-k}(z) - F_{n-\ell-2, 2-k}(z) \Big)\\
&=& (-1)^{n-1} (k-1) \big(\sum_{\ell=0}^{n-1} F_{n -1 -\ell, 2-k}(z) (\ol{c}_{n,k, \ell} - \frac{1}{k-1} \ol{c}_{n, k, \ell-1} \Big).
\end{eqnarray*}
The connection equations, after complex conjugation, state
$$
\ol{c}_{n, k, \ell} = \frac{1}{k-1} \ol{c}_{n, k, \ell-1} + \ol{c}_{n-1, k, \ell}.
$$
Substituting these in the last equation yields
$$
\xi_k \wt{G}_{n,k}(z) = (-1)^{n-1} (k-1) \sum_{\ell=0}^{n-1} F_{n-1-\ell, 2-k}(z) \ol{c}_{n-1, k , \ell} =(k-1) \wt{F}_{n-1, 2-k}(z),
$$
which verifies one  ``ramp" relation.

Using next Proposition \ref{lem:xiDelta_arbWeight_Taylor} (4) we obtain
\begin{eqnarray*}
\xi_k \wt{F}_{n,2-k}(z) &=& \xi_{2-k} \Big( (-1)^n\sum_{\ell =0}^n  \ol{c}_{n,k,\ell} F_{n-\ell, 2-k}(z)\Big)
= (-1)^n \sum_{\ell=0}^n  c_{n, k, \ell} \,\xi_{2-k}(F_{n-\ell, 2-k}(z) )\\
&=& (-1)^{n} \sum_{\ell=0}^n c_{n, k,\ell} (-1)^{n-\ell} G_{n-\ell, k}(z) 
= \big(\sum_{\ell=0}^{n} c_{n,k, \ell} (-1)^{\ell} G_{n-\ell, k}(z) 
= \wt{G}_{n,k}(z),
\end{eqnarray*}
which verifies  the other ``ramp" relation.
\end{proof}

\begin{remark}
There is freedom of choice in choosing 
the boundary values $c_{n, k, n+1}$ for
$n \ge 0.$
We call the choice $c_{n, j, n+1} =0$ for all $n$ the
 {\em zero boundary conditions.} Table \ref{tab71} presents
 value for the case of weight $k=2$.
 The main diagonal in this table are   Catalan numbers $c_{n,2,n} =\frac{1}{n+1} {{2n}\choose{n}},$
 cf. Stanley \cite{Stan2}, \cite{Stan15}.
One   gets another  elegant solution with the choice $c_{n, k, n+1} = {2n+1 \choose n+1}$
which corresponds to  $c_{n,k,n}= {2n \choose n}$;
we call these {\em binomial boundary conditions}.
In the binomial  case one obtains
$c_{n, k, \ell} = \frac{1}{(k-1)^{\ell}} {n+ \ell \choose n},$
as we will check  when proving Theorem \ref{thm:arbitraryweight}.
Table \ref{tab72} below gives values for weight $k=2$.

\end{remark}


\begin{table}[h]\centering
\renewcommand{\arraystretch}{.85}
\begin{tabular}{|r|r|r|r|r|r|r|r|r|r|}
\hline
\multicolumn{1}{|c|}{$n \backslash \ell$} &
\multicolumn{1}{c|}{$0$} &
\multicolumn{1}{c|}{$1$} &
\multicolumn{1}{c|}{$2$}&
\multicolumn{1}{c|}{$3$} &
\multicolumn{1}{c|}{$4$} &
\multicolumn{1}{c|}{$5$} &
\multicolumn{1}{c|}{$6$} &
\multicolumn{1}{c|}{$7$}  &
\multicolumn{1}{c|}{$8$} 
\\ \hline
0 &          1 &    0 &       &             &         &       &       &         &     \\ \hline
 1 &         1 &    1 &     0&            &         &       &       &           &   \\ \hline
 2 &         1 &    2 &    2&        0    &         &       &       &          &      \\  \hline 
 3 &         1&    3 &    5&        5   &     0    &       &       &           &      \\ \hline
 4 &        1&    4 &     9 &     14    &     14  &    0  &       &             &       \\ \hline
 5 &        1 &    5&    14&    28     &     42   &    42   &   0   &          &     \\ \hline
 6 &        1 &    6&    20&    48    &       90    &    132  &    132  &    0      &        \\ \hline
 7 &       1 &    7&      27&    75  &   165      &   297  &      429  &   429   &    0           \\ \hline
 
\end{tabular}
\medskip
\caption{Values of $c_{n, 2, \ell}$, with $0$ boundary conditions}
\label{tab71}
\end{table}

\begin{table}[h]\centering
\renewcommand{\arraystretch}{.85}
\begin{tabular}{|r|r|r|r|r|r|r|r|r|r|}
\hline
\multicolumn{1}{|c|}{$n\backslash \ell$} &
\multicolumn{1}{c|}{$0$} &
\multicolumn{1}{c|}{$1$} &
\multicolumn{1}{c|}{$2$}&
\multicolumn{1}{c|}{$3$} &
\multicolumn{1}{c|}{$4$} &
\multicolumn{1}{c|}{$5$} &
\multicolumn{1}{c|}{$6$} &
\multicolumn{1}{c|}{$7$} &
\multicolumn{1}{c|}{$8$} 
\\ \hline
0 &          1 &    1 &       &             &         &       &       &        &    \\ \hline
 1 &         1 &    2 &     3&            &         &       &       &          &   \\ \hline
 2 &         1 &    3 &    6&        10    &         &       &       &           &     \\ \hline
 3 &         1&    4 &    10 &        20    &     35    &       &       &        &       \\ \hline
 4 &        1&    5 &     15 &     35    &     70    &    126  &       &        &              \\ \hline
 5 &        1 &    6 &    21&    56     &     126    &    252   &   462    &         &       \\ \hline
 6 &        1 &    7 &    28&   84    &     210    &    462   &    924   &    1716      &         \\ \hline
 7 &       1 &    8&       36&    120   &   330      &   792   &    1716   &   3432  &    6435           \\ \hline
 
\end{tabular}
\medskip
\caption{Values of $c_{n, 2, \ell}$, with ${2n+1 \choose n+1}$ boundary conditions}
\label{tab72}
\end{table}

\newpage
\section{Proofs of Theorems \ref{thm:wt0VectorSpace}, \ref{thm:wt2VectorSpace}, and \ref{thm:ladder}}\label{sec:Ladder}\label{sec:Proofs}

In this section we prove the three main theorems of Section \ref{sec:Intro}. 

%
%
 
\begin{proposition}\label{prop:dimensionBound1}
For weight $k=0$ the  space $V_0^{1/2}(0)$  is one dimensional
and consists of constants.
The spaces satisfy for all $n\ge 0$ the equality
$$V_0^{n+\frac{1}{2}}(0)= V_0^{n+1}(0).$$

For integer $n \ge 1$ the space $V_0^{n}(0)$ is at  most $n$ dimensional. 
\end{proposition}
\begin{proof}
We consider the case $k=0$. The space $V_0^{1/2}(0)=M_0$ is the space of weight
$0$ holomorphic modular forms, which  has dimension $1$ and
is spanned by the constant functions (Theorem \ref{thm:classicalDimension}).
We next check that the space $V_0^1(0)$ is also generated by the constant functions,
so that $V_0^{1} (0)= V_0^{1/2}(0).$ 
First, we observe that
every  $f(z) \in V_0^1(0)$ has a Fourier expansion of the form 
$$
f(z) = \sum_{n =1}^{\infty} b_{-n} e^{-2\pi i n \ol{z}} + (b_0 y + a_0) + 
\sum_{n = 1}^{\infty} a_n e^{2\pi i n z}  
$$
for some sequences of complex numbers $a_n$ and $b_n$.
This fact follows starting from the Fourier series
given in Lemma \ref{lem:FourierExp1Harmonic},
 observing that 
$\Gamma(1, y) =\int_y^{\infty} e^{-t}dt = e^{-y}$
whence 
$$
\Gamma(1, 4 \pi |n| y) e^{2 \pi i nz} =
\Gamma(1, 4 \pi |n|y) e^{-2\pi n y} e^{2\pi i n x} = e^{2 \pi n y} e^{2 \pi i n x} = e^{2 \pi in \ol{z}} \quad 
\mbox{for} \quad  n \le -1.
$$

Applying $\xi_0$ to this form,
by Lemma \ref{lem:subspaceXi} we obtain $\xi_0f(z) \in V_2^1(0)$, 
and  computation of $\xi_0$ on the Fourier expansion above yields
$$ 
\xi_0 f(z) =\overline{b_0} +   \big( \sum_{n = 1}^{\infty} \overline{b_{-n}} (4 \pi n) e^{2\pi i n z}\big).
$$
This Fourier expansion certifies that $\xi_0(f(z))$ is a weight $2$
holomorphic modular  form for $\SL(2, \Z)$. There are no nonzero 
holomorphic modular  forms by a classical result
(Theorem \ref{thm:classicalDimension}), whence $\xi_0f(z) =0$ and
$\xi_0\(V_0^1(0)\) = \{0\}$.
It follows that 
$f(z) = \sum_{n = 0}^{\infty} a_n e^{2\pi i n z}$
is a holomorphic modular form, so
 $f(z) \in V_0^{1/2}(0)$, and  
 $V_0^{1/2}(0)= V_0^1(0)$.

We next show for all $n \ge 1$ the equality of vector spaces
$$
V_0^{n+\frac{1}{2}}(0) = V_0^{n+1}(0).
$$
Clearly $V_0^{n+1/2}(0) \subset V_0^{n+1}(0)$.
Suppose for a contradiction that  the inclusion is strict, so there exists  $f(z) \in V_0^{n+1}(0)$
and $f(z) \not\in V_0^{n+1/2}(0)$. Then 
 $g(z) := \Delta^n(f)(z) \in V_0^{1}(0)$ but $g(z) \not\in V_0^{1/2}(0)$,
 which contradicts $V_0^{1/2} (0)= V_0^{1}(0)$. The equality follows.
 
 We now show $\dim( V_0^{n+1}(0)) \le \dim(V_0^{n}(0)) +1$ by induction on
 $n \ge 0$. The base case $n=0$ is established. 
 We prove the induction step by contraction. Let $\dim(V_0^n(0)) =m$
 and suppose to the contrary that $\dim(V_0^{n+1}(0)) \ge m+2$.
 Since the constants are in $V_0^1(0)$ we may choose $m+2$
 linearly independent functions $f_0(z), ..., f_{m+1}(z)$ in $V_0^{n+1}(0)$
  with $f_0(z) =1$
 being constant. Now $\{ \Delta_0(f_i)(z) : \, 1 \le i \le m+1\}$ all lie in
 $V_0^n(0)$ 
 so must be linearly dependent, say
 $\sum_{i=1}^m \alpha_ i \Delta_0(f_i) =0$. Thus
 $g(z) := \sum_{i=1}^m \alpha_i f_i(z)  $ satisfies
 $\Delta_0( g(z) ) = \sum_{i=1}^n \alpha_i \Delta_0(f_i) =0$
 whence $g(z) \in V_0^1(0)$. This forces $g(z)=\alpha_0 $ to be constant,
 which gives a nontrivial linear relation
 $ \sum_{i=1}^{m+1} \alpha_i f_i(z) = \alpha_0 f_0(z)$
 contradicting our assumption of linear independence.
 The induction step is proved.
\end{proof}

%
%
 
\begin{proposition}\label{prop:dimensionBound2}
For weight $k=2$ the  space $V_2^{\frac{1}{2}}(0)$  is 
trivial and the space $V_2^1(0)$ is one dimensional,
spanned by $\htE_2(z, 0)$. 
The spaces satisfy for all $n\ge 0$ the equality
$$V_2^{n}(0)=  V_2^{n+\frac{1}{2}}(0).$$
For integer $n \ge 0$ the space $V_2^{n}(0)$ is at  most $n$ dimensional. 
\end{proposition}
\begin{proof}
It is a classical result that the  space $V_2^{1/2}(0) = M_2$ of holomorphic weight $2$
forms is trivial.  (Theorem \ref{thm:classicalDimension}).
The space $V_2^{1}(0)$ has dimension at least $1$ since 
$\Delta_2(\htE_2(z, 0))=0$ by Theorem \ref{th38}(3).
Furthermore  $\xi_2( \htE_2(z, 0)) = \frac{1}{2}$ is  a constant function\footnote{
The constant can be found using Proposition \ref{pr71} as $s \to 0$ in
$\xi_2( \htE_2(z, s))=(\overline{s} +1) (\overline{s} \htE_0(z, - \overline{s})),$
using the simple pole of $\htE_0(z, s)$ at $s=0$ with residue $-\frac{1}{2}$.}.

Now  $\xi_2(V_2^1(0)) \subset V_0^1(0)$
by Lemma \ref{lem:subspaceXi}, and $V_0^1(0)$ consists of constant
functions.
If $ \xi_2(V_2^1(0))$ were of dimension $2$ or greater, then
its image under $\xi_2$ would have a linear dependence, which
we could use to produce a nonzero function $f(z) \in V_2^{1}(0)$
with $g(z) =\xi_2 f(z) =0$.  By Lemma \ref{lem:FourierExp1Harmonic}
every  $f(z) \in V_2^1(0)$ has a Fourier expansion of the form 
$$
f(z) =  \sum_{n =1}^{\infty} b_{-n} \Gamma(-1, 4 \pi |n|y) e^{-2\pi i n z} +
( \frac{b_0}{y} + a_0) + \sum_{n = 1}^{\infty} a_n e^{2\pi i n z} .
$$
for some sequences of complex numbers $a_n$ and $b_n$. 
A calculation using Lemma \ref{lem:FourierExp1Harmonic} 
and Lemma \ref{le54} (1) shows
that $g(z) = \xi_2 f(z)$ has Fourier expansion
$$
g(z) = \xi_2 f(z)= - \sum_{n=1}^{\infty}\frac{ \overline{b_{-n}} }{4 \pi n} e^{2\pi i n z}=0.
$$
We conclude that all $b_n=0$, which says that
$f(z)= \sum_{n = 0}^{\infty}  a_n e^{2\pi i n z}$ 
is a weight $2$ holomorphic modular form, so $f(z)=0$. 
Thus $\dim( V_2^1(0)) =1$ and it is generated by any nonzero scalar
multiple of $\htE_2(z, 0)$.

We next show for all $n \ge 1$ the equality of vector spaces
$$
V_2^{n}(0) = V_2^{n+\frac{1}{2}}(0).
$$
Clearly $V_2^{n}(0) \subset V_2^{n+1/2}(0)$.
Suppose for a contradiction that  the inclusion is strict, so there exists  $f(z) \in V_2^{n+1/2}(0)$
and $f(z) \not\in V_2^{n}(0)$. Then 
 $g(z) := \Delta^n(f)(z) \in V_2^{1/2}(0)$ is nonzero, otherwise it is in $V_2^n(0)$,
 which contradicts the fact that $V_2^{1/2}(0)$ is $\{0\}$.
  The equality follows.
 
 Finally one may show by induction on $n$ that
  $\dim( V_2^{n+1}(0)) \le \dim(V_2^{n}(0)) +1$ holds
  for all $n \ge 0$ exactly as in Proposition \ref{prop:dimensionBound1}.
\end{proof}

%
%

\begin{proof}[Proof of Theorems \ref{thm:wt0VectorSpace} and \ref{thm:wt2VectorSpace}]
An upper bound $\dim(V_k^n(0)) \le n$
 for $k=0$ follows from Proposition \ref{prop:dimensionBound1},
 which shows also $V_0^{n}(0)= V_0^{n-\frac{1}{2}}(0).$
An upper bound $\dim(V_k^n(0)) \le n$
for $k=2$
follows from Proposition \ref{prop:dimensionBound2},
which shows $V_2^n (0)= V_2^{n+\frac{1}{2}}(0)$
It remains to show a  matching lower bound  $\dim(V_k^n(0) ) \ge n$
for $k=0, 2$ and to give a basis. 
 The Eisenstein space  $E_k^n(0) \subset V_k^n(0)$ for all $k \in 2\ZZ$,
 so the required lower bound follows from the assertion $\dim(E_k^n(0)) = n$,
 given in Corollary \ref{cor86}, which shows $V_k^n(0)= E_k^n(0)$. 
 Corollary \ref{cor86} now supplies the asserted bases
 for $V_k^n(0)$ for $k=0, 2$.
  \end{proof}

%
%

\begin{proof}[Proof of Theorem \ref{thm:ladder}]
We  apply Theorem \ref{thm:76} with $k=2$.
We need only to prove that the particular choice
$$
c_{n,2, \ell} = {n+\ell \choose \ell}
$$
satisfies the hypotheses of that result.
We have
$ c_{n,2,0} = {n \choose 0} =1$ and $c_{n, 2, -1} = {n-1 \choose -1} =0$.
Since $k=2$ the  connection equation \eqref{rec-zero} asserts
$c_{n, 2, \ell}= c_{n, 2,\ell-1} + c_{n-1, 2, \ell}$
which is exactly the binomial coefficient recursion
$$
{n+\ell \choose \ell} = {n-1 \choose \ell -1} + {n-1 \choose \ell}.
$$
The result follows. In this case  the other boundary condition states
$c_{n, 2, n+1} = {2n+1 \choose n+1}.$
\end{proof}

\section{Proof of Theorem \ref{thm:arbitraryweight} and Theorem \ref{thm:arbitrary-ramp}}\label{sec:Proofs2}

It remains to  treat the arbitrary even integer weight case, excluding $k=0$ or $2$.
The proof of Theorem \ref{thm:arbitraryweight} is similar to that of Theorems \ref{thm:wt0VectorSpace}, \ref{thm:wt2VectorSpace}, and \ref{thm:ladder}. 

%
%

\begin{proof}[Proof of Theorem \ref{thm:arbitraryweight}]
In following proof we let the weight parameter  $k \ge 4$, and the 
case of weights $ \le -2$ are represented as   $2-k$.

We first establish assertion (2) for harmonic depth $m=1$.
Suppose we are given a  weight $k \ge 4$
form $f(z) \in V_k^1(0)$. It  has a Fourier expansion of the form 
$$
f(z) =  \sum_{n=1}^{\infty} b_{-n} \Gamma(1-k, 4\pi \abs{n} y) e^{-2\pi i n z} +
( b_0 y^{1-k} + a_0) + \sum_{n=1}^{\infty} a_n e^{2\pi i n z}, 
$$
see Lemma \ref{lem:FourierExp1Harmonic}. 
Now $\xi_kf (z) \in V_{2-k}^1(0)$ by Lemma \ref{lem:subspaceXi},
and a calculation as in Lemma \ref{prop:liftOfCuspForms} we have
$$
\xi_k f(z) = (1-k) \,\overline{b_0}
-\sum_{n=1}^{\infty}   \overline{b_{-n}} (4 \pi n)^{k-1} e^{2\pi i n z}.
$$
Thus $\xi_k f(z)$ is a  holomorphic modular form, and by  Theorem
 \ref{thm:classicalDimension} it is identically $0$.  It follows that  its Fourier
 series has  $b_n=0$ for all $n \ge 0$,
 and we conclude that 
 $f(z)$ is itself holomorphic. Therefore $V_k^1(0) = V_k^{\frac{1}{2}}(0) = M_k.$
  Now by  Proposition \ref{prop:classicalFreelyGenerated} 
$ M_k= E_k^1(0) + S_k$,
where $E_k^1(0)$ contains the weight holomorphic $k$ Eisenstein series, which by
\eqref{holo-ES}
which is one dimensional.
Moreover, $S_k$ is the space of weight $k$ cusp forms. 
This establishes the assertion  (2)  for $m=1$. i.e. for $V_k^1(0)$ for $k \ge 4$.

Next we establish the assertion (1) for $V_{2-k}^1(0)$ for harmonic depth $m=1$.  
Let $f(z)\in V_{2-k}^1(0)$. Then 
$\xi_{2-k} f (z)\in V_k^1(0) = E_k^1(0) + S_k$,
whence  we have
$$
\xi_{2-k} f(z) = c_E E_k(z,0) + g(z)
$$
 where $g(z) \in S_k$ is a cusp form.
 Proposition \ref{pr71} shows that for weights $2-k \le 0$
that the function  $F_{0, 2-k}(z)= \dhtE_{2-k}(z, 0)$ has
$$
\xi_{2-k}F_{0, 2-k}(z) = \xi_{2-k} \dhtE_{2-k}(z, 0)=
\dhtE_k(z, 0)= 
\frac{k}{2}(\frac{k}{2}-1) \pi^{-\frac{k}{2}}\Gamma(k)E_k(z, 0).
$$
Now  $\frac{k}{2}(\frac{k}{2}-1) \pi^{-\frac{k}{2}}\Gamma(k) \ne 0$ since 
we assumed $k \ge 4$. 
Thus there is  a constant $C$ such that 
$h(z) = f (z)- C F_{0, 2-k}(z)$ has
$$
\xi_{2-k}(h(z)) = \xi_{2-k} (f (z)- C F_{0, 2-k}(z)) 
= g(z).
$$
Now $F_{0, 2-k} \in V_{2-k}^{1}(0)$ by Proposition \ref{lem:xiDelta_arbWeight_Taylor} (5).
so  $h(z) := f (z)- C F_{0, 2-k}(z) \in V_{2-k}^1(0)$.
Since $h(z)$ is  a  preimage of a cusp form under $\xi_{2-k}$ 
 Proposition \ref{prop:liftOfCuspForms} implies that $g(z) =0$.
Therefore, 
$f(z)$ is a multiple of $F_{0, 2-k}(z)$, and we have established that 
 $V_{2-k}^{1}(0)$ is one-dimensional, spanned by $F_{0, 2-k}(z)$,
We know that $V_{2-k}^{1/2}(0) = M_k = \{0\},$
because there are no holomorphic modular forms of negative weight,
hence  $V_{2-k}^{1/2}(0) \ne V_{2-k}^{1}(0)$. 
Finally if $f(z) \in V_{2-k}^{3/2}(0)$, then $\Delta_{2-k}(f)(z) \in V_{2-k}^{1/2}(0) = \{0\}$,
so $f(z) \in V_{2-k}^{1}(0)$ and $V_{2-k}^{1}(0)= V_{2-k}^{3/2}(0)$.
This establishes assertion (1) for $m=1$ for weights $2-k \le -2$.

We have proved assertions (1) and (2) for $m=1$ 
and complete the proof by  induction on $m \ge 1$.  For assertion (1),
the  induction hypothesis says that  
$V_{2-k}^m(0)$ is $m$-dimensional.  
Now Lemma \ref{lem:xiDelta_arbWeight} shows that $V_{2-k}^{m+1}(0)$
is at least $m+1$ dimensional, since $F_{m, k}(z)  \in V_{2-k}^{m+1}(0)\setminus V_{2-k}^m(0)$. 
Suppose that there 
is a form  $f(z) \in V_{2-k}^{m+1}(0)$ linearly independent
 from the set $\{ F_{0, 2-k}(z), \ldots, F_{m, 2-k}(z)\}$. 
Then $\Delta_{2-k} f(z), \Delta_{2-k} F_{0, k}(z), \cdots, \Delta_{2-k} F_{m, k}(z) \in V_{2-k}^m(0)$
are linearly dependent, so there  exist $C, c_0, c_1, \cdots, c_m$ so that 
$$\Delta_{2-k} \( Cf (z)+ c_0 F_{0, k}(z) + \cdots c_m F_{m, k}\) = 0.$$
It now follows that 
$Cf (z)+ c_0 F_{0,k}(z) + \cdots c_m F_{m, k}(z) = D F_{0,k}(z)$ 
holds for some constant $D$, contracting linear independence.  
Thus $V_{2-k}^{m+1}(0)$ has dimension $m+1$.
It is easy to deduce that $V_{2-k}^{m+1}(0) = V_{2-k}^{m+ \frac{3}{2}}(0)$
by applying $(\Delta_k)^{m+1}$ to a given $f(z) \in V_{2-k}^{m+ \frac{3}{2}}(0)$.

The proof of the induction step for assertion  (2) is similar.
We must also use  the result of Proposition 
\ref{prop:liftOfCuspForms}  which shows there are no $1$-harmonic lifts of cusp forms.
\end{proof}

%
%

\begin{proof}[Proof of Theorem \ref{thm:arbitrary-ramp}]
This result follows  from Theorem \ref{thm:76}.
It remains only to show for $k \ge 4$ that 
$$
c_{n,k, \ell} = \frac{1}{(k-1)^{\ell}}{ n+ \ell \choose \ell}
$$
satisfies the connection equations  of that result.
It is easy to check that $c_{n,k,0}=1$ and $c_{n, k, -1}=0$. The
connection equations \eqref{rec-zero} become the identity
$$
 \frac{1}{(k-1)^{\ell}}{ n+ \ell \choose \ell}= \frac{1}{k-1} (\frac{1}{(k-1)^{\ell -1}}{n+ \ell -1 \choose \ell-1})
 + \frac{1}{(k-1)^\ell}{n+ \ell -1 \choose \ell}.
$$
\end{proof}

\subsection*{Acknowledgments}
The authors thank the reviewer for helpful comments.
The first author thanks the Clay Foundation for support
as a Clay Senior Fellow at ICERM, an NSF-supported institute,
where some work on this
paper was done.


\end{document}